\documentclass[11pt]{amsart}
\usepackage{a4wide, amsmath, amssymb}
\usepackage[usenames]{color}
\newtheorem{thm}{Theorem}[section]

\newtheorem{lem}[thm]{Lemma}
\newtheorem{prop}[thm]{Proposition}
\theoremstyle{definition}
\theoremstyle{definition}
\theoremstyle{definition}
\theoremstyle{definition}

\newtheorem{rem}[thm]{Remark} \numberwithin{equation}{section}

\newcommand{\F}{\mathcal{F}}

\newcommand{\K}{\mathcal{K}}
\newcommand{\R}{\mathbb R}
\newcommand{\E}{\mathbb E}
\newcommand{\T}{\mathbb T}

\def\P{\mathbb P}
\def\1{\mathbb I}

\def\e{\epsilon}
\def\g{\gamma}
\def\l{\lambda}

\def\o{\omega}

\def\ov{\overline}

\definecolor{vert}{rgb}{0,0.4,0}

\begin{document}
\title[Large time behavior of weakly coupled systems]{
{Some results on the large time behavior of weakly 
coupled systems of first-order
Hamilton-Jacobi equations}}

\author{Vinh Duc Nguyen}
\address{IRMAR, INSA de Rennes, 35708 Rennes, France} 
\email{vinh.nguyen@insa-rennes.fr}

\begin{abstract}
Systems of Hamilton-Jacobi equations arise naturally when we study optimal control problems
 with pathwise deterministic
trajectories with random switching. In this work, we are interested in the large time behavior
of  weakly coupled systems of first-order
Hamilton-Jacobi equations in the periodic setting. First results have been obtained by Camilli-Loreti-Ley and the author (2012) and Mitake-Tran (2012) under quite strict conditions. Here, we use a PDE approach to extend the convergence result
proved by Barles-Souganidis (2000) in the scalar case. This result permits us to treat general cases, for instance, systems of nonconvex Hamiltonians and systems of strictly convex Hamiltonians. We also obtain some other convergence results under different assumptions. These results give a clearer view on the large time behavior for systems of Hamilton-Jacobi equations.
\end{abstract}

\maketitle



\section{Introduction}

\subsection{Statement of the problem and recalls of the existing results}\label{intro1}

We study the large time behavior of systems of Hamilton-Jacobi equations
\begin{eqnarray}\label{HJEi}
&& \left\{
\begin{array}{ll}
\displaystyle\frac{\partial u_i}{\partial t}
+  H_i(x, Du_i )+\sum_{j=1}^{m}d_{ij}u_j=0 &
(x,t)\in \T^N\times (0,+\infty),\\[3mm]
u_i(x,0)=u_{0i}(x) & x\in\T^N,
\end{array}
\right.
 i=1,\dots,m,
\end{eqnarray}
where $\T^N$ is the $N$-dimensional flat torus, $H_i\in C(\T^N \times \R^N),~~u_{0i}\in C(\T^N)$, and the coupling $D=(d_{ij})_{i,j=1,\dots,m}$ is  monotone, i.e.,
\begin{eqnarray}\label{S3}
d_{ii}\geq 0, \quad d_{ij}\leq 0 \ {\rm for} \ i\not= j\quad {\rm and}  \quad
\sum_{j=1}^{m}d_{ij}=0~~\text{for  all $i$}
\end{eqnarray}
and the system is fully coupled in some sense which will be precised later (see \ref{irr}).

The aim of this work is to improve the first results obtained by
Camilli-Ley-Loreti and the author \cite{clln12} and Mitake-Tran \cite{mt12}
and, more generally, to generalize to systems of the form~\eqref{HJEi},
the existing results for the case of a single Hamilton-Jacobi equation
\begin{equation}\label{HJE}
\left\{
  \begin{array}{ll}
\frac{\partial u}{\partial t}+  H(x, Du)=0,
   & (x,t)\in \T^N\times (0,+\infty),\\[5pt]
 u(x,0)=u_{0}(x) & x\in\T^N.
  \end{array}
\right.
\end{equation}
It is convenient to start by recalling the existing results for~\eqref{HJE}, and the first results of ~\cite{clln12,mt12} for systems and then to turn to our results.

The large time behavior has been extensively investigated using both PDE methods and dynamical approaches.
The desired result is to find a unique constant $c\in\R, $ the so-called critical value or ergodic constant, and a solution $v$ of the stationary equation
\begin{eqnarray}\label{ergo_sing}
H(x, Dv )=c \quad \text{in } \T^N
\end{eqnarray}
such that 
\begin{eqnarray}\label{cvHJ}
u(x,t)+ct \to v(x)~~\text{uniformly as $t$ tends to infinity}.
\end{eqnarray}

\noindent The first results are of Fathi \cite{fathi98} and Namah-Roquejoffre \cite{nr99} where the convexity of Hamiltonians plays a key role. The result of \cite{fathi98} was proved for uniformly convex Hamiltonians, i.e., there exists a constant $\alpha >0$ such that 
$D^2_{pp}H(x,p) \ge \alpha I,~~\text{for all}~~(x,p) 
\in \T^N \times \mathbb{R}^N.$

In Davini-Siconolfi~\cite{ds06}, the authors extended this result to strictly convex Hamiltonians, i.e., for any $0<\lambda<1,~~x\in \T^N,~~p \neq q$, we have
\begin{eqnarray}\label{strictconvex}
H(x,\lambda p+(1-\lambda)q) < \lambda H(x,p)+(1-\lambda)H(x,q).
\end{eqnarray}

In \cite{nr99}, the result was proved for Hamiltonians of the form 
\begin{eqnarray*}
H(x,p)=F(x,p)-f(x),\quad {\rm with} \ F(x,p)\geq F(x,0)=0,
\end{eqnarray*}
and $F$ is coercive and convex with respect to $p.$ 
In this framework, the set
\begin{eqnarray}\label{Fscalar}
\F_{scalar}=\{x_0\in\T^N:\,f(x_0)={\rm min}_{x \in \T^N}f(x)\}
\end{eqnarray}
plays a crucial role.
It appears to be a {\it uniqueness set} for the stationary 
equation~\eqref{ergo_sing}, i.e., the solution of \eqref{ergo_sing} 
is uniquely characterized by its value on this set.
\smallskip

Barles and Souganidis~\cite{bs00} succeeded in relaxing a bit the convexity condition on $H$.
Under the two following sets of assumptions, the first generalizing ~\cite{nr99} and the second generalizing ~\cite{fathi98},
they obtain the convergence~\eqref{cvHJ}. From now, we always denote ${\rm dist}(x,K)$ the usual distance function which is defined to be  plus infinity if $K$ is empty, and $\psi(\eta)$ is a positive constant depending only on $\eta$.
Assume that \eqref{ergo_sing} is solved for $c=0$ and we introduce the assumptions on $H$ (in general, the assumptions are made on $H-c$)
\begin{eqnarray}\label{H5}
&& \left\{
\begin{array}{ll}
\text{$(i)~~|H(x,p)-H(y,p)| \le m(|x-y|(1+|p|))$,}\\
\text{\quad \quad where $m$ is a nonnegative function such that $m(0^{+})=0$.}\\
\text{$(ii)~~$There exists a, possibly empty, compact subset $K$ of $\mathbb{T}^{N}$ such that} \\  
\text{$\quad (a)~~H(x,p) \ge 0$ ~~\text{on} ~~$K \times \mathbb{R}^{N}$,}\\
\text{$\quad (b)~~$If ${\rm dist}(x,K) \ge \eta >0$, $H(x,p+q)\ge \eta$ and $H(x,q)\le 0$},\\
\text{\quad \quad \quad we have $H_p(x,p+q)p-H(x,p+q)\ge \psi(\eta)>0$}.
\end{array}
\right.
\end{eqnarray}
or
\begin{eqnarray}\label{H7}
&& \left\{
\begin{array}{l}
\text{$(i)~~$The function $p \mapsto H(x,p)$ is differentiable a.e. in $x \in \T^{N}$,}\\
\text{$(ii)~~$There exists a, possibly empty, compact set $K$ of $\mathbb{R}^{N}$ such that}\\
\text{$\quad (a)~~H(x,p) \ge 0$ on $K \times \mathbb{R}^{N}$ ,}\\
\text{$\quad (b)$~~If $H(x,p) \ge \eta > 0$ and ${\rm dist}(x,K) \ge \eta $,}\\
\text{\quad \quad \quad we have $H_{p}(x,p)p-H(x,p) \ge \psi(\eta)>0$}.
\end{array}
\right.
\end{eqnarray}


Let us mention Fathi \cite{fathi10}, Roquejoffre \cite{roquejoffre01} for other related results in the periodic setting. Some of these results have been also extended beyond the periodic setting, see Barles and Roquejoffre \cite{br06}, Ishii \cite{ishii08},
Ichihara and Ishii \cite{ii09}, and for problems with periodic boundary conditions, see for instance 
Mitake~\cite{mitake08b, mitake08a, mitake09}. We refer also the reader to
Ishii~\cite{ishii06, ishii09} for an overview.
\smallskip

In the case of systems, we are interested in finding an ergodic constant vector 
$(c_1,\dots,c_m)\in \R^m$ and a function $(v_1,\dots,v_m)$ such that
\begin{eqnarray}\label{HJEista}
H_i(x, Dv_i )+\sum_{j=1}^{m}d_{ij}v_j=c_i,~~x \in \T^N,~~i=1,\dots,m
\end{eqnarray}
and, for all $i=1,\dots,m,$
\begin{eqnarray}\label{tosim_con}
u_i(x,t)+c_it \to v_i(x)~~\text{uniformly as $t$ tends to infinity,}
\end{eqnarray}
 where $u$ is the solution of \eqref{HJEi}.
\smallskip

First results for the system~\eqref{HJEi} were obtained in~\cite{clln12} and~\cite{mt12}. It was proved that \eqref{HJEista} has a solution $((c_1,\dots,c_1),v) \in \R^m \times W^{1,\infty}(\T^N)^m$ for convex, coercive Hamiltonians. By the way, we extend this result to systems with a coupling matrix depending on $x$ (see Theorem \ref{ergo3}). The convergence \eqref{tosim_con} was obtained when
$H_i(x,p)=F_i(x,p)-f_i(x)$ 
satisfying the 
same properties as in~\cite{nr99} (see above and~\eqref{coco}). The set $\F_{scalar}$ defined in \eqref{Fscalar} was replaced by
\begin{eqnarray}\label{same}
\F=\{x_0\in \T^N:f_i(x_0)=\min_{x \in \T^N}f_j(x)~~\text{for all $i, j=1,\dots,m$}\}.
\end{eqnarray} 
One of the main (and restrictive) assumption was to suppose that $\F \neq  \emptyset$, i.e., {\it the $f_i$'s attain their minima at the same point with the same value.} Then, they proved that $\F$ is also a uniqueness set to derive the convergence. In this case, the interpretation of the convergence 
in terms of optimal control of pathwise deterministic trajectories with random switching (see \cite[Section 6]{clln12} and Appendix) is clear. One should rather drive the trajectories to a common minimum of the $f_i$'s
since these latters play the role of the running costs of the control problem.

The extension of such a result to the case $\F = \emptyset$
was the most challenging issue
which was addressed in~\cite{clln12} and one of the motivation of this paper.

\subsection{Main results}\label{main}
We need some conditions on the coupling. The irreducibility is a classical assumption when dealing with systems of PDEs. Roughly speaking, when the coupling is irreducible, the system is not separated into many smaller systems. 
\begin{eqnarray}\label{irr}
&&
\begin{array}{l}
\text{$D$ is irreducible: if $\forall~ \mathcal{I}\varsubsetneq \{1,\cdots ,m\},$ there exist $i\in \mathcal{I}$, $j\not\in  \mathcal{I}$ such that $d_{ij}\not= 0$.}
\end{array}
\end{eqnarray}
We also need the following assumption
\noindent
\begin{eqnarray} \label{row_nonzero}
\text{ $D$ has a nonzero coefficients line: if there exists $i$ such that }
d_{ij}\not= 0,~~\text{$\forall  j$}.
\end{eqnarray}
We assume that \eqref{HJEi} has a unique solution $u$ such that $u+ct\in W^{1,\infty}(\T^N \times (0,\infty))^m$, and \eqref{HJEista} has a solution $(c,v)\in \R^m \times W^{1,\infty}(\T^N )^m$. By the change of function $u(x,t) \to u(x,t)+ct$ and replacing $H_i$ with $H_i-c_i$, we may assume that $c=0$. We now introduce our assumption which replaces \eqref{H7} (notice that when $c \neq 0$, the assumptions have to be written for $H_i-c_i$). 

For $i=1,\dots ,m$
\begin{eqnarray}\label{H10}
&& \left\{
\begin{array}{ll}
\text{$(i)~~$The function $p \mapsto H_i(x,p)$ is differentiable a.e. in $x \in \T^{N}$.}\\[2mm]
\text{$(ii)~~(H_i)_p~~p-H_i\ge 0$ for  a.e. $(x,p) \in \T^N \times \mathbb{R}^{N}$,}\\[2mm]
\text{$(iii)~~$There exists a, possibly empty, compact set $K$ of $\mathbb{T}^{N}$ such that}\\
\text{$\quad(a)~~H_i(x,p) \ge 0$ on $K \times \R^{N}$ ,}\\
\text{$\quad(b)~~$If $H_i(x,p) \ge \eta > 0$ and ${\rm dist}(x,K) \ge \eta$, then $(H_i)_p~~p-H_i \ge \psi(\eta)>0$.}
\end{array}
\right.
\end{eqnarray}
This  assumption seems to be a natural extension 
of~\eqref{H7} to systems. The reasons to introduce such an assumption is explained in Section \ref{bs1}. But let us make some comments and give some motivations. In general, one does not know the exact value of $c$. However \eqref{H10} can be applied for some important cases: strictly convex Hamiltonians and some nonconvex Hamiltonians.

We are now able to state our main result,
the proof of which is given in Section~\ref{bs1}
\begin{thm}\label{mainresult}{\rm (Main convergence result)}
Suppose that $D$ satisfies \eqref{irr} and \eqref{row_nonzero}, $H_i$ satisfies \eqref{H10}. Then, the solution $u=(u_1,\dots,u_m) \in W^{1,\infty}(\T^N \times (0,\infty))^m$ of~\eqref{HJEi} converges uniformly to a solution $(v_{\infty 1 },\dots,v_{\infty m })$ of \eqref{HJEista} with $(c_1,\dots,c_m)=0$. 
\end{thm}
In general, the existence of a Lipschitz continuous solution of~~\eqref{HJEi} is  followed by the coercivity of the Hamiltonians (see Theorem~\ref{ergo3}). 

Comparing~\eqref{H10} and~\eqref{H7},
we realize that~\eqref{H10} (ii) is the only additional assumption, which is crucial
in the proof of~\eqref{main_rough4} (see Theorem~\ref{IM1} $(i)$). 
The important examples satisfying Theorem~~\ref{mainresult} are
nonconvex Hamiltonians. The following example is drawn from \cite{bs00}: $H_i(x,p)=\psi_i(x,p)F_i(x,\frac{p}{|p|})-f_i(x)$,
where $f_i \in C(\T^N)$ is nonnegative, $F_i\in C(\T^N \times \R^N)$ is
strictly positive, and $\psi_i(x,p)=|p+q_i(x)|^2-|q_i(x)|^2$ with $q_i\in C(\T^N; \R^N)$. Moreover, we assume that
$K=\{x\in \T^N : f_i(x)=|q_j(x)|=0~~\text{for all $i, j=1,\dots,m$}\} \neq \emptyset.$
Under these conditions, we can show that $c=(0,\dots,0)$ (see \cite[Lemma 4.1, Theorem 4.2]{clln12} for the same calculations) and since
\begin{eqnarray*}
(H_i)_p(x,p)~~p-H_i(x,p)=|p|^2F_i(x,\frac{p}{|p|})+f_i(x),
\end{eqnarray*} 
it is clear that $H_i$ satisfies \eqref{H10} with $K$ defined as above.

Sometimes, we cannot use directly \eqref{H10}. The following example is very typical from mechanic. Notice that $f_1,~~f_2$ may not attain their minima at the same point.
\begin{equation}\label{diff_mini}
\left\{
  \begin{array}{ll}
\frac{\partial u_1}{\partial t}+|Du_1|^2+u_1-u_2=f_1(x),
   & \\[5pt]
 \frac{\partial u_2}{\partial t}+|Du_2|^2+u_2-u_1=f_2(x).
  \end{array}
\right.
\end{equation}
However, and it is one of the main achievement of this paper,
we will see below how our main result Theorem \ref{mainresult} can be applied 
to give a full answer to this problem, see~Theorem \ref{exist_smoo}. Roughly speaking, in some cases, Theorem \ref{mainresult} will not be applied directly but after a change of Hamiltonians. The following result is an important application of Theorem~~\ref{mainresult}
for systems of strictly convex Hamiltonians.
\begin{thm}\label{exist_smoo}
Suppose that $D$ satisfies 
\eqref{irr} and \eqref{row_nonzero},
$H_i$ satisfies \eqref{strictconvex} and coercive in $p$ uniformly in $x\in \T^N$. 
Then, there exists
$c=(c_1,...,c_1)$ and  a solution $v \in W^{1,\infty}(\T^N)^m$ of \eqref{HJEista} 
such that
$u+ct \to v $ in $C(\T^N)^m$, where $u$ is the solution of \eqref{HJEi}.
\end{thm}
\noindent We provide hereunder some formal ideas so that the reader can see how we use Theorem \ref{mainresult}. The general proof is given in Section \ref{strict}.

We fix a Lipschitz continuous subsolution $V$ of \eqref{HJEista} and the associated ergodic constant $c=(c_1,\dots,c_1)$ (see Theorem~\ref{ergo3}). We assume that $V$ is {\em $C^1$} to perform a formal proof. Set $w_i=u_i+c_1t-V_i$, it is clear that $w$ is the bounded solution of
\begin{eqnarray*}
\displaystyle\frac{\partial w_i}{\partial t}
+ H_i(x, DV_i+Dw_i )-H_i(x, DV_i )-g_i(x)+\sum_{j=1}^{m}d_{ij}w_j=0,~~i=1,\dots,m.
\end{eqnarray*} 
Here,  $g_i(x):=-H_i(x, DV_i )-\sum_{j=1}^{m}d_{ij}V_j+c_1 \ge 0$ and $g_i \in C(\T^N)$ for all $i=1,\dots,m$
since $V$ is a {\em $C^1$} subsolution of \eqref{HJEista}. 
We define the new Hamiltonians $G_i(x,p)=H_i(x, p+DV_i )-H_i(x, DV_i )-g_i(x).$ 
Since $H_i$ is coercive, the solutions $u$ of \eqref{HJEi} and $v$ of \eqref{HJEista} are Lipschitz continuous.
We aim at applying Theorem~\ref{mainresult} to get the desired convergence, it is then sufficient to check
that~\eqref{H10} holds with $K=\emptyset$ and $p$ bounded. This is done thanks to the strict convexity of $H$.


\noindent This theorem extends the result of \cite{ds06} to systems.
It also gives a full answer to the Eikonal type Hamiltonians
case of~\cite{clln12,mt12}: when the Hamiltonians are
strictly convex, one has the convergence 
regardless $\mathcal{F}$ defined in \eqref{same} is empty or not. 
In particular, we can prove the convergence for~\eqref{diff_mini}.

We learnt very recently that Mitake and Tran~\cite{mt12d}
obtained the same result as Theorem~\ref{exist_smoo} using a  dynamical approach which corresponds, in the case
of systems, to the method of~\cite{ds06}. Here, the result is a particular case
of a general PDE approach.

\subsection{Miscellaneous convergence results} \label{miscell}
We obtain some particular results under different assumptions on the
Hamiltonians. These results are not completely covered by the main result
and bring to light some interesting phenomena. 
\smallskip

\subsubsection{Hamiltonians of Eikonal type.} \label{eikonal-intro}
We focus on the setting of Namah and Roquejoffre \cite{nr99}, i.e., for the Hamiltonians of the form
\begin{eqnarray*}
H_i(x,p)=F_i(x,p)-f_i(x), \quad x\in \T^N, p\in\R^N,
\end{eqnarray*}
where $F_i\in C(\T^N\times \R^N),~~f_i\in C(\T^N)$, and for $i=1,\cdots,m$,
\begin{eqnarray}
   & & \text{$F_i(x,\cdot)$ is convex, coercive in $p$ uniformly in $x\in \T^N$
 and }
  F_i(x,p)\geq F_i(x,0)=0\; . \label{coco}
\end{eqnarray}
We can extend the results of \cite{mt12,clln12} when $\mathcal{F}$
defined in \eqref{same} is replaced by 
\begin{eqnarray}\label{samemin}
\mathcal{S}:=\{x_0 \in \T^N,f_i(x_0)=\min_{x \in \T^N}f_i(x),~~\text{for all $i$}
\} \neq \emptyset.
\end{eqnarray}
This latter condition means that the $f_i$'s attain their minima at the same point but {\em their~value at this point may be different.} 

The main point in this result is that it applies to convex Hamiltonians without the requirement to be strictly convex (in this better case, we can apply Theorem \ref{exist_smoo}).  The idea is that we can find the explicit
formula for $c_i$ together with a constant solution of the ergodic problem. 
It is then possible to transform the system such that the new one satisfies the conditions of Theorem \ref{mainresult}.
The problem is open if $D$ depends on $x$.
\begin{thm}\label{largenew}
Suppose that $D$ satisfies \eqref{irr}-\eqref{row_nonzero}, and 
\eqref{coco}-\eqref{samemin} hold. Then there exist 
$c=(c_1,...,c_1) \in \R^m$ and $u_\infty \in W^{1,\infty}(\T^N)^m$ solution of \eqref{HJEista} such that $u+ct \to u_\infty $ in $C(\T^N)^m$, where $u$ is the solution of \eqref{HJEi}.
\end{thm}

\subsubsection{The case when all Hamiltonians are identical} 
In \cite{mt12}, the authors obtained the convergence of the solution of \eqref{HJEi} with $m=2$, $H_1=H_2$ satisfying \eqref{H5}. The proof 
is based on the ideas used for single equations, so it only works under the set of conditions \eqref{H5}. In this setting, we can give a very easy proof  of the convergence of solutions of \eqref{HJEi} for a very wide class of Hamiltonians.
In fact, we observe that the convergence of solutions of systems
is actually {\em inherited from the convergence of solution of the corresponding scalar equation.} We have the following theorem.

\begin{thm}\label{identical}
Suppose that $D(x)=(d_{ij}(x))_{1 \le i,j \le m}$, where the $d_{ij}$'s are continuous functions and for any $x \in \T^N$, $D(x)$ satisfies \eqref{nonzero}. Assume that $H_i=H$ for all $i=1,\dots,m$ and the solution of equation \eqref{HJE} converges as $t$ tends to infinity. Then any solution $u(.,t)\in BUC(\T^N \times (0,\infty))^m$ of \eqref{HJEi} converges to a solution $v$ of \eqref{ergo_sing}.
\end{thm}
\smallskip


\subsection{Organization of the paper}
In Section \ref{preli}, preliminary results for the coupling are given,
the ergodic problem for coercive Hamiltonians is solved and 
basic properties of the solutions like existence, uniqueness and Lipschitz 
regularity are presented. 
Next sections are devoted to the proofs of the theorems stated in the introduction. Section \ref{bs1} contains the proof of Theorem \ref{mainresult}, which
is the most technical and involved part. Next, we give the proof of Theorem \ref{exist_smoo} in Section \ref{strict}. Since the ideas of this proof are based on the ideas used in the proof Theorem \ref{mainresult}, we strongly recommend the reader to read it
after reading Section \ref{bs1}. 
The proof of Theorem \ref{largenew} is given in Section \ref{nr} and the proof of Theorem \ref{identical} 
is given in Section~\ref{bs2}. 

\smallskip

\noindent Notations: Since we only work with viscosity solutions in this paper, we will drop the term ``viscosity'' hereafter. We denote by $C(\T^N)^m$ ($BUC(\T^N)^m$, $W^{1,\infty}(\T^N)^m$) the set of functions $u=(u_1,\dots,u_m)$, where $u_i:\T^N \to \R$ is continuous (bounded uniformly continuous, Lipschitz continuous respectively) for all $i=1,\dots,m$.
\smallskip

\noindent{\sc Acknowledgments.} I would like to thank O. Ley without whom I would have never finished this paper. I thank 
G.~Barles who gave me a lot of useful advice to improve the
first version of this work and M.~Camar-Eddine for useful discussions. 
I thank H. Mitake and H. V. Tran for letting the author know about
their recent work. Last but not least, I thank the two referees for their careful reading and their suggestions for this paper.
\smallskip

\section{Some preliminary results}\label{preli}
In this section, we state the results for the coupling matrices which may depend on $x$. The reason is that we can solve the ergodic problem- Theorem \ref{ergo3} in that case which is an interesting result by itself. 
\subsection{Preliminaries on coupling matrices.} 

We begin with a key property of irreducible matrices.
\begin{lem} \label{lem-rang}{\rm (\cite{clln12})}
Suppose that $D(x)=(d_{ij}(x))_{1 \le i,j \le m}$, where the $d_{ij}$'s are continuous functions and for any $x \in \T^N$, $D(x)$ satisfies \eqref{irr}. Then, for all $x\in \T^N,$ $D(x)$ is degenerate of rank $m-1,$
${\rm ker}(D(x))={\rm span}\{(1, \cdots, 1)\}$
and the real part of each nonzero complex eigenvalue of
$D(x)$ is positive. Moreover
there exists a positive continuous function
$\Lambda=(\Lambda_1,\dots,\Lambda_m): \T^N \to\R^m$ such that
$\Lambda(x) >0$ and $D(x)^T\Lambda(x) =0$ for all $x\in \T^N .$
\end{lem}

\subsection{Ergodic problem for systems of Hamilton-Jacobi equations with coercive Hamiltonians.}\label{ge}
The classical result for first-order Hamilton-Jacobi equations is due to 
Lions-Papanicolaou-Varadhan~\cite{lpv86}.
In general, the ergodic problem \eqref{ergo_sing} is solved in the following way: we first prove a gradient bound for the regularized equation $ \l v^\l+ H(x,Dv^\l)=0$. Due to the coercivity of $H$, this gradient bound is independent of $\l$. Since $v^\l$ may not be bounded, we make a change of function $w^\l=v^\l-v^\l(x_0)$ in the equation. It follows $|Dw^\l| \le L$ and $w^\l$ is uniformly bounded thanks to the compactness of $\T^N$. And then, the requirements of Ascoli's theorem are fulfilled.
However, for systems with a $x$-dependent coupling matrix, the change of variable $w_i^\l=v_i^\l-v_i^\l(x_0)$ leads to additional terms in the system which are difficult to control. That is why we required the coupling matrix to be independent of $x$ to prove \cite[Theorem 4.3]{clln12}, see also ~\cite{mt12}. Here we can overcome this difficulty by noting that we only need the uniform bound for $v_i^\l-v_1^\l(x_0)$ (see the  estimate~\eqref{keyes}) and hence the change of variable $w_i^\l=v_i^\l-v_1^\l(x_0)$ turns out to be a good one since it does not appear additional terms. Let us point out that we cannot use the gradient bound to obtain the bound for $v_i^\l-v_1^\l(x_0)$.
\begin{thm}\label{ergo3}
Suppose that $D(.)$ satisfies the conditions of Lemma \ref{lem-rang}, $H_i(x,.)$ is coercive in $p$ uniformly in $x$. Then, there is a solution $((c_1,...,c_1),v)\in\R^m\times W^{1,\infty}(\T^N)^m$ to~\eqref{HJEista} with $(c_1,...,c_1)\in {\rm ker}\,D(x)$ for all $x \in \T^N$. Moreover, $(c_1,...,c_1)$ is unique in ${\rm ker}\,D$.
\end{thm}
The proof is given in Appendix.

\subsection{Maximum principle and compactness properties of the solution.}
The following results are proven in \cite{clln12}. 
\begin{prop}\label{PrComp}
Suppose that $D(.)$ satisfies the conditions of Lemma \ref{lem-rang} and either the $H_i$'s are coercive in $p$ or $u$ (or $v$) is Lipschitz. Let $u$, $v$ be
a subsolution and a supersolution of ~\eqref{HJEi}, respectively . We have
\begin{itemize}
  \item [$(i)$] for any $t\geq 0,$
\begin{eqnarray}\label{comp-max}
\mathop{\rm max}_{1\leq i\leq m}\mathop{\rm sup}_{\T^N}\, (u_i(\cdot,t)-v_i(\cdot,t))
\leq \mathop{\rm max}_{1\leq i\leq m}\mathop{\rm sup}_{\T^N}
\left(u_{i} (\cdot,0)-v_{i} (\cdot ,0)\right)^+.
\end{eqnarray}
  \item [$(ii)$]for any $u_0\in C(\T^N)^m$, there exists a unique continuous solution 
of \eqref{HJEi} which admits $u_0$ as a initial condition.
\end{itemize}
\end{prop}
Using the existence of solutions of the ergodic problem proved in Theorem \ref{ergo3}, we obtain
\begin{prop}\label{PrReg}{\rm (\cite{clln12})}
Under the assumptions of Theorem \ref{ergo3}, let $u_0\in W^{1,\infty}(\T^N)^m$ and $u$ be the solution of \eqref{HJEi} with
initial data $u_0$. Then, there exists constant $C$ such that
\begin{align*}
   &|u(x,t)+ct|\le C, & x \in\T^N,\,  t\in [0,\infty),\\
   &|u(x,t)-u(y,s)|\leq C(|x-y|+|t-s|),
& x,y\in\T^N,\,  t,s\in [0,\infty),
\end{align*}
where the vector $c=(c_1,\dots,c_1)$ is given in Theorem \ref{ergo3}.
\end{prop}


\section{Proof of the main result, Theorem \ref{mainresult}.}\label{bs1}
\smallskip

\subsection{Formal proof. }

In order to explain the new hypothesis \eqref{H10}, we redo the formal proof for the single equation made in \cite{bs00} and then try to mimic it for systems.

In~\cite{bs00}, the authors first show that
\begin{eqnarray}\label{main_rough} 
\min_{x \in \T^N}~~{\frac{\partial u}{\partial t}(x,t)}\to 0~~\text{as $t$ tends to infinity}.
\end{eqnarray}
The main consequence is that the $\omega$-limit set of $\{u(.,t),~~t \ge 0\}$ 
defined by
\begin{equation*}
    \o(u)=\{\psi \in C(\T^N):\,
\exists t_n\to\infty \text{ such that } \lim_{n\to\infty}u(.,t_n)=\psi\}
\end{equation*}
contains only subsolutions of \eqref{ergo_sing}. 

This fact with the compactness of $\T^N$ are enough to prove \eqref{cvHJ}. 
To prove \eqref{main_rough}, we perform a change of function of the form $ \exp(w)=u$.
It turns out that $w$ solves
\begin{eqnarray}\label{inter_sing}
\frac{\partial w}{\partial t}+F(x,w,Dw)=0,~~\text{with}~~F(x,w,p)= {\rm exp}(-w)H(x,{\rm exp}(w)p ),
\end{eqnarray}
and $F$ inherits the properties of $H$
\begin{eqnarray}\label{H9}
&& \left\{
\begin{array}{ll}
\text{There exists a, possibly empty, compact set $\mathcal{K}$ of $\mathbb{T}^{N}$ such that}\\[2mm]
\text{if}~~F(x,w,p) \ge \eta >0~~\text{and ${\rm dist}(x,\mathcal{K}) \ge \eta $},~~\text{then}~~F_w(x,w,p)\ge \psi(\eta)>0~~\text{a.e.}
\end{array}
\right.
\end{eqnarray}
An application of the maximum principle yields that $t \mapsto \min_{x \in \T^N}~~{\frac{\partial w}{\partial t}(x,t)}$ is nonincreasing so it converges. If the limit is nonnegative, we obtain easily the convergence of $w(x,t)$ as $t \to \infty$. Otherwise, there exists some $\eta>0,~~t_0>0$ such that for all $t \ge t_0$
\begin{eqnarray}\label{main_rough2}
\min_{x \in \T^N}~~{\frac{\partial w}{\partial t}(x,t)}\le -\eta.
\end{eqnarray}
Set $z=\frac{\partial w}{\partial t}$ and $m(t)=\min_{x \in \T^N}~~{z(x,t)}:=z(x_t,t)$. Differentiating \eqref{inter_sing} with respect to $t$, we obtain
\begin{eqnarray*}
\frac{\partial z}{\partial t}+F_w(x,w,Dw)z+F_p.Dz=0,~~\text{hence $m'+F_w(x_t,w,Dw)m=0$}.
\end{eqnarray*}
Using \eqref{H9} and \eqref{main_rough2}, we get
\begin{eqnarray*}
m'+\psi(\eta)m \ge 0,~~\text{thus}~~m(t)\ge m(t_0)e^{-\psi(\eta)(t-t_0)}
\end{eqnarray*}
Letting $t$ tends to infinity yields a contradiction with \eqref{main_rough2}.
\smallskip

Now, let us try to mimic the above formal proof for systems through the typical example
\begin{eqnarray*}
&& \left\{
\begin{array}{ll}
\frac{\partial u_1}{\partial t}+  H_1(x, Du_1)+u_1-u_2=0, &\\[3mm]
\frac{\partial u_2}{\partial t}+  H_2(x, Du_2)+u_2-u_1=0,
\end{array}
(x,t)\in \T^N\times (0,+\infty),
\right.
\end{eqnarray*} 
where the $H_i$'s satisfy \eqref{H7}. After the change of function $\exp(w_i)=u_i$, $w_i$ satisfies
\begin{eqnarray}\label{inter11_sing}
&& \left\{
\begin{array}{ll}
\frac{\partial w_1}{\partial t}+F_1(x,w_1,Dw_1 )+1-{\rm exp}(w_2-w_1)=0, &\\[3mm]
\frac{\partial w_2}{\partial t}+F_2(x,w_2,Dw_2 )+1-{\rm exp}(w_1-w_2)=0,
\end{array}
(x,t)\in \T^N\times (0,+\infty),
\right.
\end{eqnarray}
where $F_i(x,w,p)= {\rm exp}(-w)H_i(x,{\rm exp}(w)p )$. 
The goal is to prove that 
\begin{eqnarray}\label{main_roughsy} 
\min_{x \in \T^N,~~i=1,2}~~{\frac{\partial w_i}{\partial t}(x,t)}\to 0~~\text{as $t$ tends to infinity}.
\end{eqnarray}
An application of the maximum principle yields that $t \mapsto \min_{x \in \T^N,~~i=1,2}~~{\frac{\partial w_i}{\partial t}(x,t)}$ is nonincreasing so it converges. If the limit is nonnegative, we obtain easily the convergence of $w_i(x,t)$ as $t \to \infty$. Otherwise, there exists some $\eta>0,~~t_0>0$ such that for all $t \ge t_0$
\begin{eqnarray}\label{main_rough9}
\min_{x \in \T^N,~~i=1,2}~~{\frac{\partial w_i}{\partial t}(x,t)}\le -\eta.
\end{eqnarray}
Set $z_i=\frac{\partial w_i}{\partial t}$ and assume that 
\begin{eqnarray}\label{main_rough7}
m(t)=\min_{x \in \T^N,~~i=1,2}~~{z_i(x,t)}:=z_1(x_t,t).
\end{eqnarray} 
This fact, \eqref{inter11_sing} and \eqref{main_rough9} only give
\begin{eqnarray}\label{main_rough5}
F_1(x_t,w_1,Dw_1 )+1-{\rm exp}(w_2-w_1)(x_t,t)\ge \eta.
\end{eqnarray} 
We see that \eqref{H9} cannot apply here, since we cannot control the additional term $1-{\rm exp}(w_2-w_1)(x_t,t)$ using only the information given by \eqref{main_rough7}.

Surprisingly, under \eqref{H10} $(ii)$, we will show that we are able to choose a $x_t \in \T^N$ in \eqref{main_rough7} such that
\begin{eqnarray}\label{main_rough4}
m(t)=\min_{x \in \T^N,~~i=1,2}~~{z_i(x,t)}:=z_1(x_t,t)=z_2(x_t,t).
\end{eqnarray} 
This fact, \eqref{inter11_sing} and  \eqref{main_rough9} give us {\em two} inequalities
\begin{eqnarray*}
&& \left\{
\begin{array}{ll}
F_1(x_t,w_1,Dw_1 )+1-{\rm exp}(w_2-w_1)(x_t,t)\ge \eta, &\\[3mm]
F_2(x_t,w_2,Dw_2 )+1-{\rm exp}(w_1-w_2)(x_t,t)\ge \eta.
\end{array}
\right.
\end{eqnarray*}
So if $w_1(x_t,t) \le w_2(x_t,t)$ for instance, we have $F_1(x_t,w_1,Dw_1 )\ge \eta.$ Now, we can continue the proof accordingly. The fact that $z_1$ and $z_2$ attain their minima at the same point in \eqref{main_rough4} is a key idea to break the difficulty when passing from equations to systems.
\smallskip


\subsection{Proof of the main result. }
We first state a partial convergence result for \eqref{HJEi}.
\begin{lem}\label{par_con_resu}
For any $x \in K$ defined in \eqref{H10}, the solution $u(x,t)$ of \eqref{HJEi} converges as $t \to \infty$.
\end{lem}
The proof of this lemma can be deduced from Step 2 in the proof of Lemma \ref{IM2} $(ii)$. So we skip it for shortness.


Define $\exp(w_i):=u_i$. It is clear that $w$ solves
\begin{eqnarray}\label{inter1}
\frac{\partial w_i}{\partial t}+F_i(x,w_i,Dw_i )+\sum_{j=1}^{m}d_{ij}{\rm exp}(w_j-w_i)=0,~~i=1,\dots,m,
\end{eqnarray} 
with $F_i(x,w,p)= {\rm exp}(-w)H_i(x,{\rm exp}(w)p )$. We get the properties of the $F_i$'s which are inherited from the $H_i$'s
\begin{eqnarray}\label{H13}
&& \left\{
\begin{array}{ll}
\text{$(i)~~(F_i)_w(x,w,p)\ge 0$ for a.e. $(x,w,p)$,}\\[2mm]
\text{$(ii)~$There exists a, possibly empty, compact set $\mathcal{K}$ of $\mathbb{T}^{N}$ such that}\\
\text{\quad if}~~F_i(x,w,p) \ge \eta >0~~\text{and ${\rm dist}(x,\mathcal{K}) \ge \eta $},~~\text{then}~~(F_i)_w(x,w,p)\ge \psi(\eta)>0~~\text{a.e.}
\end{array}
\right.
\end{eqnarray}
For any $\eta>0$ and  $\varphi\in BUC(\T^N\times [0,\infty))$, we define
\begin{eqnarray}\label{newfunc}
P_{\eta}[\varphi](t)&=&\sup_{x \in \T^N,~~s \ge t}[\varphi(x,t)-\varphi(x,s)-2\eta(s-t)].
\end{eqnarray}
From Proposition \ref{PrReg} and Ascoli theorem, we obtain easily the relative compactness of $\{u(.,.+t),~~t \ge 0\}$ in $C(\T^N)$ and hence of $\{w(.,.+t),~~t \ge 0\}$.  So we can choose a sequence $t_n\to +\infty$ such that $(w(\cdot,t_n+\cdot))_{n}$ converges uniformly to
some function $v\in W^{1,\infty}(\T^N\times [0,\infty))^m$. By the stability result (\cite{ bcd97,barles94, el91}),  $v$ 
is still a viscosity solution of~\eqref{inter1}. We state the key estimates on the functions $P_{\eta}[v_i]$'s.


\begin{thm}\label{IM1}
Consider the system \eqref{inter1}, where $D$ satisfies \eqref{irr} and \eqref{row_nonzero}. Assume that $F_i\in C(\T^N\times \R\times \R^N)$ satisfies \eqref{H13}, and $w_i(.,t) \in W^{1,\infty}(\T^N\times (0,+\infty))$ converges as $t \to \infty$ in $\mathcal{K}$. We have

$(i)~~P_{\eta}[v_i](t)=c(\eta)$, where $c(\eta)$ depends only on $\eta$. Moreover, the $P_{\eta}[v_i](t)$'s attain their maximum at the same point $(x_t,s_t)$ for all $i$, i.e.,
\begin{eqnarray}\label{tothink}
&&P_{\eta}[v_i](t)=v_i(x_t,t)-v_i(x_t,s_t)-2\eta(s_t-t)\\
&=&P_{\eta}[v_j](t)=v_j(x_t,t)-v_j(x_t,s_t)-2\eta(s_t-t)~~\text{for all $i,j=1,\dots,m$ and $t>0$.}\nonumber
\end{eqnarray}
$(ii)~~c(\eta)=0$ for any $\eta >0$.
\end{thm}
\begin{rem}
Let us emphasize that \eqref{tothink} is the key fact to prove Theorem~\ref{mainresult}. It is the rigorous statement of \eqref{main_rough4}. To prove this, we need \eqref{H13}$(i)$ but not \eqref{H13}$(ii)$. The important condition \eqref{H13}$(ii)$ only plays a role in the proof of part $(ii)$ of Theorem \ref{IM1}.
\end{rem}
We first give the proof of Theorem \ref{mainresult} 
using Theorem~\ref{IM1}. And then we prove  Theorem~\ref{IM1}.

\begin{proof}[Proof of Theorem \ref{mainresult}]
{\it Step 1.} Since the solution $u$ of  \eqref{HJEi} is bounded, we can assume, by adding a big enough positive constant on the initial conditions if needed, that 
\begin{eqnarray*}
M\ge u_i(x,t) \ge 1,~~\text{$x \in \T^N,~~$ $t>0$, $i=1,\dots,m$}.
\end{eqnarray*}

Set
\begin{eqnarray}\label{chgt-var123}
\exp(w_i(x,t))=u_i(x,t)~~\text{for all $i=1,\dots,m$ and $(x,t) \in \T^N \times (0,\infty)$}.
\end{eqnarray} 
Then $w_i \in W^{1,\infty}(\T^N\times (0,+\infty))$  solves \eqref{inter1}
with $F_i(x,w,p)= {\rm exp}(-w)H_i(x,{\rm exp}(w)p ).$ We can check that $F_i$ 
satisfies \eqref{H13} with $\mathcal{K}:=K$ given in \eqref{H10}. Moreover, Lemma \ref{par_con_resu} gives the convergence of $u_i(x,t)$ and hence of $w_i(x,t)$, for all $x \in K$. Now, all the conditions of Theorem \ref{IM1} are fulfilled.

{\it Step 2.} From Theorem \ref{IM1} $(ii)$, we have
\begin{eqnarray*}
v_i(x,t)-v_i(x,s)-2\eta(s-t)\le0,~~\text{for all $i,~~s \ge t$ and $x \in \T^N$}.
\end{eqnarray*}
Letting $\eta$ tend to 0, we obtain
\begin{eqnarray*}
v_i(x,t)-v_i(x,s)\le 0,~~\text{for all $i$, $x \in \T^N,~~s \ge t \ge 0$}.
\end{eqnarray*}
Therefore, $v_i$ is nondecreasing in $t$, so $v_i(x,t) \to v_{i \infty}(x)$ as $t$ tends to infinity. And since $v_i$ is Lipschitz continuous, $v_{i \infty}$ is also Lipschitz continuous. Therefore, we can use Dini's Theorem to deduce that $v_i(x,t) \to v_{i \infty}(x)$ uniformly in $x$ as $t$ tends to infinity.

{\it Step 3.} The uniform convergence of $(w(\cdot,t_n+\cdot))_{n}$ to $v\in W^{1,\infty}(\T^N\times [0,+\infty))$ yields
\begin{eqnarray*}
o_n(1)+v_i(x,t)\le w_i(x,t+t_n) \le o_n(1)+v_i(x,t)~~\text{in $\T^N \times (0,\infty)$},~~\text{for all $i$}.
\end{eqnarray*}

Taking $\limsup^*$ and $\liminf_*$ with respect to $t$ both sides of the above estimate, we obtain
\begin{eqnarray*}
o_n(1)+v_{i \infty}(x)\le \liminf_{t\to +\infty}\phantom{ }_*\, w_i(x,t)\le \limsup_{t\to +\infty}\phantom{ }^*\,w_i(x,t) \le o_n(1)+v_{i \infty}(x)
\quad x\in\T^N,\text{ for all $i.$}
\end{eqnarray*}
Letting $n$ tend to infinity, we derive
\begin{eqnarray*}
\liminf_{t\to +\infty}\phantom{ }_*\, w_i(x,t)=\limsup_{t\to +\infty}\phantom{ }^*\, w_i(x,t)=v_{i \infty}(x), \quad x\in\T^N,\text{ for all $i,$}
\end{eqnarray*}
which yields the uniform convergence of $w_i(.,t)$ to $v_{i \infty}$ 
in $\T^N$ as $t$ tends to infinity. 

{\it Step 4.} By stability, $v_{i \infty}$ is a solution of \eqref{inter1}. 
Therefore, $u_i(.,t)={\rm exp}(w_i(.,t))$ converges uniformly, as $t$ tends to infinity,
to ${\rm exp}(v_{i \infty})$  which 
is a solution of \eqref{HJEista} with $(c_1,\dots,c_m)=0$ under our assumption. It ends the proof of Theorem \ref{mainresult}.
\end{proof}
\subsection{Proof of Theorem \ref{IM1}.}
We need the following result, the proof of which is given in subsection \ref{lemme_auxi}.
\begin{lem}\label{IM2}
Under the assumptions of Theorem \ref{IM1}, we have

$(i)~~N_{\eta,i}:={\exp(P_{\eta}[w_i])}$ are subsolution of
\begin{eqnarray}\label{A10}
N'_{\eta,i}+\sum_{j \neq i}^{m}d_{ij}\rho_{ij}(t)[N_{\eta,j}(t)-N_{\eta,i}(t)]\le 0,~~i=1,\dots,m,
\end{eqnarray}
where $\rho_{ij}$ are functions defined by
\begin{eqnarray*}
\rho_{ij}(t)=
&& \left\{
\begin{array}{ll}
m_1~~\text{if $N_{\eta,i}(t)\ge N_{\eta,j}(t)$},\\
m_2~~\text{if $N_{\eta,i}(t)< N_{\eta,j}(t)$},
\end{array}
\right.
\end{eqnarray*}
for all $i\neq j$, $t>0,$ and $m_1\le m_2$ are positive constants.

$(ii)~~$ The $N_{\eta,i}(t)$'s converge to the same limit as $t$ tends to infinity.
\end{lem}
\begin{proof}[Proof of part $(i)$ of Theorem \ref{IM1}]

{\it Step 1.} We notice that $P_{\eta}[\varphi]$ introduced in \eqref{newfunc} is nonnegative, bounded uniformly continuous thanks to the fact that $\varphi \in W^{1,\infty}(\T^N\times (0,+\infty))$. Since $(w(\cdot,t_n+\cdot))_{n}$ converges uniformly to $v$ as $t_n \to \infty$, we easily obtain that $P_{\eta}[v_i](t)=c(\eta)$ for all $i$ and $t$, where $c(\eta)$ depends only on $\eta$ thanks to Lemma \ref{IM2}. Fix $\tau>0$ and choose $i \in \{1,\dots,m\}$ such that the $i${\it th} row is nonzero, this is possible thanks to~\eqref{row_nonzero}. If $P_{\eta}[v_i](\tau)= 0$, then we finish the proof since we can choose $s_\tau=\tau$ and any $x_\tau$ in \eqref{tothink} to fulfill the requirement. We then assume that $P_{\eta}[v_i](\tau)> 0$ and that $P_{\eta}[v_i](\tau)$ attains its maximum at $x_\tau,~~s_\tau$. It is worth mentioning at this stage of the proof that $x_\tau,~~s_\tau$ depend on $i$.

{\it Step 2.} Consider for $x,y \in \T^N,~~t \in (0,\infty)$ and $s \ge t$ the test function
\begin{eqnarray*}
\Psi(x,y,t,s)=v_i(x,t)-v_i(y,s)-2\eta(s-t)-|x-x_\tau|^2-|t-\tau|^2-|s-s_\tau|^2-\frac{|x-y|^2}{2\epsilon^2}.
\end{eqnarray*}
Assume that $\Psi$ achieves its maximum over $\T^N \times \T^N \times  \{(t,s) / 0 \le t \le s\}$ at  $(\bar{x},\bar{y},\bar{t},\bar{s})$. We obtain some classical estimates when $\epsilon \to 0$,
\begin{eqnarray}\label{Es2}
&& \left\{
\begin{array}{ll}
\Psi(\bar{x},\bar{y},\bar{t},\bar{s}) \to P_{\eta}[v_i](\tau),~~\frac{|\bar{x}-\bar{y}|^2}{2\epsilon^2} \to 0,~~
(\bar{x},\bar{s},\bar{t}) \to (x_\tau,s_\tau,\tau),\\

v_i(\bar{x},\bar{t})-v_i(\bar{y},\bar{s}) \ge c(\eta),\\
\bar{s}>\bar{t} \quad \text{since} \quad c(\eta)>  0.
\end{array}
\right.
\end{eqnarray}
{\it Step 3.} Since $v$ is the solution of \eqref{inter1}, we can write the viscosity inequalities
\begin{eqnarray*}
&& \left\{
\begin{array}{ll}
2(\bar{t}-\tau)-2\eta+F_i(\bar{x},v_i(\bar{x},\bar{t}),\frac{\bar{x}-\bar{y}}{\epsilon^2}+2(\bar{x}-x_\tau))+\sum_{j=1}^{m}d_{ij}{\rm exp}(v_j-v_i)(\bar{x},\bar{t})\le 0,\\
-2(\bar{s}-\tau)-2\eta+F_i(\bar{y},v_i(\bar{y},\bar{s}),\frac{\bar{x}-\bar{y}}{\epsilon^2})+\sum_{j=1}^{m}d_{ij}{\rm exp}(v_j-v_i)(\bar{y},\bar{s}) \ge 0,
\end{array}
\right.
\end{eqnarray*}
Since $v_i$ is bounded Lipschitz continuous,
$|F_i(\bar{x},v_i(\bar{x},\bar{t}),\frac{\bar{x}-\bar{y}}{\epsilon^2})
-F_i(x_\tau,v_i(\bar{x},\bar{t}),\frac{\bar{x}-\bar{y}}{\epsilon^2})|\le m(|\bar{x}-x_\tau|) \le o_\e(1),~~|F_i(\bar{x},v_i(\bar{y},\bar{s}),\frac{\bar{x}-\bar{y}}{\epsilon^2})
-F_i(x_\tau,v_i(\bar{y},\bar{s}),\frac{\bar{x}-\bar{y}}{\epsilon^2})| \le o_\e(1)$ 
thanks to the uniform continuity of $F_i$ over compact subsets. It follows from \eqref{Es2},
\begin{eqnarray}\label{Es10}
&& \left\{
\begin{array}{ll}
-2\eta+F_i(x_\tau,v_i(\bar{x},\bar{t}),\frac{\bar{x}-\bar{y}}{\epsilon^2})+\sum_{j=1}^{m}d_{ij}{\rm exp}(v_j-v_i)(\bar{x},\bar{t})\le o_\e(1),\\
-2\eta+F_i(x_\tau,v_i(\bar{y},\bar{s}),\frac{\bar{x}-\bar{y}}{\epsilon^2})+\sum_{j=1}^{m}d_{ij}{\rm exp}(v_j-v_i)(\bar{x},\bar{s}) \ge o_\e(1).
\end{array}
\right.
\end{eqnarray}

{\it Step 4.} Set $Q_{\eta,j}(\tau):=v_j(x_\tau,\tau)-v_j(x_\tau,s_\tau)-2\eta(s_\tau-\tau)$ for $j \neq i$, we have
\begin{eqnarray}\label{T1}
&&\sum_{j=1}^{m}d_{ij}{\rm exp}(v_j-v_i)(\bar{x},\bar{t})-\sum_{j=1}^{m}d_{ij}{\rm exp}(v_j-v_i)(\bar{x},\bar{s})\\
&=& \sum_{j=1,j \neq i}^{m}-d_{ij}{\rm exp}(v_j-v_i)(x_\tau,s_\tau)\big\{1-{\rm exp}[Q_{\eta,j}(\tau)-P_{\eta}[v_i](\tau)]\big\}+o_\e(1).\nonumber
\end{eqnarray}
Using \eqref{H13}$(i)$ and the fact that $v_i(\bar{x},\bar{t})-v_i(\bar{y},\bar{s})\ge P_{\eta}[v_i](\bar{t}) \ge 0$, we have
\begin{eqnarray*}
F_i(x_\tau,v_i(\bar{x},\bar{t}),\frac{\bar{x}-\bar{y}}{\epsilon^2})-F_i(x_\tau,v_i(\bar{y},\bar{s}),\frac{\bar{x}-\bar{y}}{\epsilon^2}) \ge 0.
\end{eqnarray*}
Therefore, by letting $\epsilon \to 0$ after subtracting both sides in \eqref{Es10}, we get
\begin{eqnarray*}
\sum_{j \neq i}-d_{ij}{\rm exp}(v_j-v_i)(x_\tau,s_\tau)\big\{1-{\rm exp}[Q_{\eta,j}(\tau)-P_{\eta}[v_i](\tau)]\big\}\le 0.
\end{eqnarray*}
Thanks to the choice of $i$ in Step 1 and~\eqref{row_nonzero}, we have $-d_{ij} > 0$ for $j\neq i$. And since $Q_{\eta,j}(\tau)\le P_{\eta}[v_i](\tau)$, it follows from the above inequality that $Q_{\eta,j}(\tau)=P_{\eta}[v_i](\tau)$, i.e., $P_{\eta,i}$'s attain their maximum at the same point $(x_\tau,s_\tau)$.

{\it Proof of part $(ii)$.} We only need to repeat Step 2, 3 with a refinement of Step~~4. We assume by contradiction that $P_{\eta}[v_i](\tau)=c(\eta)>0$.

{\it Step 5.} From the convergence of $w_i(x,t)$ on $\mathcal{K}$ and the definition of $v_i$'s, we have
\begin{eqnarray}\label{H14_better}
v_i(x,t)~~\text{is independent of $t$  for all $x \in \mathcal{K}$,}
\end{eqnarray}
where $\mathcal{K}$ is defined in \eqref{H13}. From part $(i)$, for any fixed $\tau>0$, there exists $x_\tau,~~s_\tau$ satisfying \eqref{tothink}. Unlike Step 1, $x_\tau,~~s_\tau$ at this stage does not depend on $i$ anymore. We next choose $k \in \{1,\dots,m\}$  such that 
\begin{eqnarray*}
v_k(x_\tau,s_\tau)=\min_{j=1,\dots,m}v_j(x_\tau,s_\tau).
\end{eqnarray*}
This gives
\begin{eqnarray*}
\sum_{j=1}^{m}d_{kj}{\rm exp}(v_j-v_k)(\bar{x},\bar{s})\le o_\e(1).
\end{eqnarray*}
Thanks to this estimate, we get from the second inequality of \eqref{Es10}
\begin{eqnarray*}
F_k(x_\tau,v_k(\bar{y},\bar{s}),\frac{\bar{x}-\bar{y}}{\epsilon^2}) \ge \eta>0~~\text{for $\epsilon$ small enough}.
\end{eqnarray*}
It follows from \eqref{H14_better}
\begin{eqnarray*}
d(x_\tau,\mathcal{K}) \ge \beta_\eta>0.
\end{eqnarray*}
Thanks to~\eqref{H13} $(ii)$ and the Lipschitz continuity of $F_k$ with respect to $w,$
we infer
\begin{eqnarray*}
&& F_k(x_\tau,u,\frac{\bar{x}-\bar{y}}{\epsilon^2}) \ge \eta>0, \quad {\rm for \ all } \  u \ge v_k(\bar{y},\bar{s}),\\
&& (F_k)_w(x_\tau,u,\frac{\bar{x}-\bar{y}}{\epsilon^2}) \ge \psi(\eta)>0, \quad\text{for almost all }u \ge v_k(\bar{y},\bar{s}).
\end{eqnarray*}
Since $v_k(\bar{x},\bar{t})\geq v_k(\bar{y},\bar{s}),$ we obtain
\begin{eqnarray*}
F_k(x_\tau,v_k(\bar{x},\bar{t}),\frac{\bar{x}-\bar{y}}{\epsilon^2})-F_k(x_\tau,v_k(\bar{y},\bar{s}),\frac{\bar{x}-\bar{y}}{\epsilon^2}) \ge \psi(\eta)(v_k(\bar{x},\bar{t})-v_k(\bar{y},\bar{s})) \ge \psi(\eta)c(\eta).
\end{eqnarray*}
As $P_{\eta}[v_k](.)=P_{\eta}[v_j](.)=c(\eta)$, it follows from \eqref{T1} that,
\begin{eqnarray*}
\sum_{j=1}^{m}d_{kj}{\rm exp}(v_j-v_k)(\bar{x},\bar{t})-\sum_{j=1}^{m}d_{kj}{\rm exp}(v_j-v_k)(\bar{x},\bar{s})
\ge o_\e(1).
\end{eqnarray*}
Therefore, by letting $\epsilon$ tend to 0 after subtracting both sides in \eqref{Es10}, we get
\begin{eqnarray*}
\psi(\eta)c(\eta)\le 0,~~\text{it is a contradiction}.
\end{eqnarray*}
\end{proof}

\subsection{Proof of Lemma \ref{IM2}.}\label{lemme_auxi}
\begin{proof}
{\it Proof of part $(i)$.}

\noindent{\it Step 1}. First of all, we show that
$P_{\eta}[w_i]$ are subsolution of
\begin{eqnarray}\label{A9}
P'_{\eta}[w_i]+\sum_{j=1,~~j \neq i}^{m}-d_{ij}\rho_{ij}(\tau)\big[1-{\rm exp}(P_{\eta}[w_j]-P_{\eta}[w_i])\big]\le 0,~~i=1,\dots,m,
\end{eqnarray}
where $\rho_{ij}$ are functions defined by
\begin{eqnarray}\label{def of rho}
\rho_{ij}(t)=
&& \left\{
\begin{array}{ll}
m_1~~\text{if $P_{\eta}[w_i](t)\ge P_{\eta}[w_j](t)$},\\
m_2~~\text{if $P_{\eta}[w_i](t)< P_{\eta}[w_j](t)$},
\end{array}
\right.
\end{eqnarray}
for all $i\neq j$, $t>0,$ and $m_1\le m_2$ are positive constants. To do so, let $\Phi \in C^1((0,\infty))$ and $\tau$ be a strict maximum point of $P_{\eta}[w_i]-\Phi$ over $[\tau-\delta,\tau+\delta]$ for some $\delta >0$. If $P_{\eta}[w_i](\tau)=0$ and since $P_{\eta}[w_i] \ge 0$, we see that $\tau$ is a minimum point of $\Phi$ and hence $\Phi'(\tau)=0$. To prove \eqref{A9}, it is enough to show that
\begin{eqnarray*}
\sum_{j=1,~~j \neq i}^{m}-d_{ij}\rho_{ij}(\tau)\big\{1-{\rm exp}[P_{\eta}[w_j](\tau)]\big\}\le 0,
\end{eqnarray*}
this is clearly true since $P_{\eta}[w_j^\delta] \ge 0$. We then assume that $P_{\eta}[w_i](\tau)> 0$ to continue.

\noindent{\it Step 2}. Consider, $x,y \in \T^N,t \in [\tau-\delta,\tau+\delta]$ and $s \ge t$ the test function
\begin{eqnarray*}
\Psi(x,y,t,s)=w_i(x,t)-w_i(y,s)-2\eta(s-t)-\frac{|x-y|^2}{2\epsilon^2}-\Phi(t).
\end{eqnarray*}
Assume that $\Psi$ achieves its maximum over $\T^N \times \T^N \times  \{(t,s) / t \le s ,t \in [\tau-\delta,\tau+\delta]$\} at  $(\bar{x},\bar{y},\bar{t},\bar{s})$. We obtain some classical estimates when $\epsilon \to 0$,
\begin{eqnarray*}
&& \left\{
\begin{array}{ll}
\Psi(\bar{x},\bar{y},\bar{t},\bar{s}) \to P_{\eta}[w_i](\tau)-\Phi(\tau),~~\frac{|\bar{x}-\bar{y}|^2}{2\epsilon^2} \to 0,\\ [1mm]
~~\text{$\bar{t} \to \tau$ since $\tau$ is a strict maximum point of $P_{\eta}[w_i]-\Phi$ in $[\tau-\delta,\tau+\delta]$},\\
w_i(\bar{x},\bar{t})-w_i(\bar{y},\bar{s}) \ge P_{\eta}[w_i](\bar{t}),~~\bar{s}>\bar{t} \quad \text{since} \quad P_{\eta}[w_i](\tau)> 0.
\end{array}
\right.
\end{eqnarray*}
\noindent{\it Step 3}. Since $w$ is the solution of \eqref{inter1}, we have
\begin{eqnarray}\label{Es20}
&& \left\{
\begin{array}{ll}
\Phi'(\bar{t})-2\eta+F_i(\bar{x},w_i(\bar{x},\bar{t}),\frac{\bar{x}-\bar{y}}{\epsilon^2})+\sum_{j=1}^{m}d_{ij}{\rm exp}(w_j-w_i)(\bar{x},\bar{t})\le 0,\\
-2\eta+F_i(\bar{x},w_i(\bar{y},\bar{s}),\frac{\bar{x}-\bar{y}}{\epsilon^2})+\sum_{j=1}^{m}d_{ij}{\rm exp}(w_j-w_i)(\bar{x},\bar{s})+o_\e(1) \ge 0.
\end{array}
\right.
\end{eqnarray}

Using \eqref{H13} and the fact that $w_i(\bar{x},\bar{t}) \ge w_i(\bar{y},\bar{s})\ge P_{\eta}[w_i](\bar{t}) \ge 0$, we get
\begin{eqnarray*}
F_i(\bar{x},w_i(\bar{x},\bar{t}),\frac{\bar{x}-\bar{y}}{\epsilon^2})-F_i(\bar{x},w_i(\bar{y},\bar{s}),\frac{\bar{x}-\bar{y}}{\epsilon^2}) \ge 0.
\end{eqnarray*}
Moreover,
\begin{eqnarray*}
&&\sum_{j=1}^{m}d_{ij}{\rm exp}(w_j-w_i)(\bar{x},\bar{t})-\sum_{j=1}^{m}d_{ij}{\rm exp}(w_j-w_i)(\bar{x},\bar{s})\\
&\ge& \sum_{j \neq i}-d_{ij}{\rm exp}(w_j-w_i)(\bar{x},\bar{s})\big\{1-{\rm exp}[P_{\eta}[w_j](\tau)-P_{\eta}[w_i](\tau)]\big\}+o_\e(1)\\
&\ge& \sum_{j \neq i}-d_{ij}\rho_{ij}(\tau)\big\{1-{\rm exp}[P_{\eta}[w_j](\tau)-P_{\eta}[w_i](\tau)]\big\}+o_\e(1),
\end{eqnarray*}
where $\rho_{ij}$ is defined as in~\eqref{def of rho}
with $$m_1=\inf_{x \in \T^N,~~s>0,~~1 \le i,j \le m}{\rm exp}(w_j-w_i)(x,s)>0,$$ $$m_2=\sup_{x \in \T^N,~~s>0,~~1 \le i,j \le m}{\rm exp}(w_j-w_i)(x,s)<\infty$$
which are well-defined thanks to the boundedness of $w_i$.

Therefore, by letting $\epsilon$ tend to 0 after subtracting both sides in \eqref{Es20}, we get: 
\begin{eqnarray*}
\Phi'(\tau)+\sum_{j=1,~~j \neq i}^{m}-d_{ij}\rho_{ij}(\tau)\big\{1-{\rm exp}[P_{\eta}[w_j](\tau)-P_{\eta}[w_i](\tau)]\big\}\le 0.
\end{eqnarray*}

{\it Proof of part $(ii)$.}
The proof is quite technical since we cannot deal directly with the discontinuity of the $\rho_{ij}$'s. The main idea is to reorder the $N_i$'s into the biggest, the second biggest ... and the smallest function. Surprisingly, the new functions satisfy a nicer system where the discontinuous functions $\rho_{ij}$ are replaced by constants. We then prove they converge to the same limit and, as a result, the $N_i$'s, which are bounded by the biggest function and the smallest function, must converge to the same limit. 
For simplicity of notations, we suppose that $m_1=1,~~m_2=2$.

{\it Step 1.} Define
\begin{eqnarray}\label{reorder}
&&R_1(t)=\max_{i \in \{1,\dots,m\}}N_i(t):=N_{i_1}(t),\\
&&R_2(t)=\max_{i \in \{1,\dots,m\}-\{i_1\}}N_i(t):=N_{i_2}(t),\nonumber\\
&&R_k(t)=\max_{i \in \{1,\dots,m\}-\{i_1,\dots,i_{k-1}\}}N_i(t):=N_{i_k}(t),~~k=2,\dots,m.\nonumber
\end{eqnarray}
We will prove at Step 3 that $R_i$ satisfies
\begin{eqnarray}\label{A11}
R'_i(t)+\sum_{j=1}^{m}d'_{ij}R_j(t)\le 0,~~i=1,\dots,m,
\end{eqnarray}
where $(d'_{ij})_{1 \le i,j \le m}$ satisfies \eqref{S3} and \eqref{row_nonzero}. 

\noindent{\it Step 2.} Call $\Lambda_i$ be the vector from Lemma \ref{lem-rang} for the coupling now is $(d'_{ij})_{1 \le i,j \le m}$. We have
\begin{lem}\label{LB7}{\rm (\cite[Lemma 5.5]{clln12})}
$\sum_{j=1}^m \Lambda_j R_j(t)$ is nonincreasing and converges as $t\to +\infty.$
\end{lem}
From~\eqref{row_nonzero}, there exists $i\in\{1,\dots,m\}$ and $\alpha >0$ such that
\begin{eqnarray}\label{strict111}
 d'_{ij}+\alpha\Lambda_j < 0
\quad j=1,\dots ,m, \ j\not= i.
\end{eqnarray}
From \eqref{A11}, we obtain that $R_i$ is a subsolution of
\begin{equation*}
R'_i(t)+(d'_{ii}+\alpha\Lambda_i)R_i
\le \alpha \Lambda_i R_i +\sum_{j\not= i}(-d'_{ij})R_j,
\qquad t\in(0,+\infty).
\end{equation*}
From the stability result, 
$\ov R_i=\limsup_{t\to +\infty}\phantom{ }^*\, R_i(t)$ is a subsolution of
\begin{equation}\label{eq528}
(d'_{ii}+\alpha\Lambda_i)\ov R_i
\le \limsup_{t\to +\infty}\phantom{ }^*\,
\{ \alpha \Lambda_i R_i+\sum_{j\not= i}(-d'_{ij})R_j\}.
\end{equation}
But,
\begin{eqnarray*}
&& \limsup_{t\to +\infty}\phantom{ }^*\,
\{ \alpha \Lambda_i R_i+\sum_{j\not= i}(-d'_{ij})R_j\}\\
&\leq &
 \alpha \limsup_{t\to +\infty}\phantom{ }^*\,\{  \sum_{j=1}^m \Lambda_j R_j\}
+
\sum_{j\not= i}
\limsup_{t\to +\infty}\phantom{ }^*\,\{ 
(-d'_{ij}-\alpha\Lambda_j )R_j
\}\\
&\leq &
 \alpha \limsup_{t\to +\infty}\phantom{ }^*\,\{  \sum_{j=1}^m \Lambda_j R_j\}
+\sum_{j\not= i}(-d'_{ij}-\alpha\Lambda_j)
\ov R_j,
\end{eqnarray*}
since \eqref{strict111} holds.
The previous inequality and~\eqref{eq528} imply
\begin{eqnarray*}
\alpha\sum_{j= 1}^m\Lambda_j \ov R_j
&=& (d'_{ii}+\alpha\Lambda_i)\ov R_i 
+\sum_{j\not= i}(d'_{ij}+\alpha\Lambda_j)
\ov R_j
\le 
\alpha \limsup_{t\to +\infty}\phantom{ }^*\,\{  \sum_{j=1}^m \Lambda_j R_j\}.
\end{eqnarray*}
Using Lemma~\ref{LB7}, it follows
\begin{eqnarray}\label{equa345}
&&\sum_{j= 1}^m\Lambda_j\ov R_j 
= \limsup_{t\to +\infty}\phantom{ }^*\,\{  \sum_{j=1}^m \Lambda_j R_j\}
\leq \sum_{j=1}^m \Lambda_j R_j(t),
\quad t\in (0,+\infty).
\end{eqnarray}
Therefore, for all $k=1,\dots,m,$ we have
\begin{eqnarray*}
\Lambda_k (R_k(t)-\ov R_k)\geq 
\sum_{j\not= k}\Lambda_j \ov R_j 
- \sum_{j\not= k} \Lambda_j R_j(t).
\end{eqnarray*}
Moreover 
\begin{eqnarray*}
\Lambda_k( \underline  R_k - \ov R_k)
&=& \liminf_{t\to +\infty}\phantom{ }_*\,\{ \Lambda_k
(R_k(t)-\ov R_k)\}\\
&\geq&  \liminf_{t\to +\infty}\phantom{ }_*\,\{
\sum_{j\not= k}\Lambda_j\ov  R_j 
- \sum_{j\not= k} \Lambda_jR_j(t)\}
\geq
\sum_{j\not= k}\Lambda_j\ov R_j
-  \limsup_{t\to +\infty}\phantom{ }^*\,\{ \sum_{j\not= k} 
\Lambda_jR_j(t)\}
\geq 0
\end{eqnarray*}
by~\eqref{equa345}. Since $\Lambda_k>0$,
we conclude that $\ov R_k\leq \underline R_k$ and hence $R_k(t) \to r_k$ when $t$ tends to infinity. Taking into account these limits, we get from the stability result for~\eqref{A11}
\begin{eqnarray*}
\sum_{j=1}^{m}d'_{ij}r_j\le 0,~~\text{and hence }\sum_{j=1}^{m}d'_{ij}r_j= 0,~~i=1,\dots,m.
\end{eqnarray*}
An application of Lemma \ref{lem-rang} shows that $r_i=r_j$ for all $i,j=1,\dots,m,$ which is the desired result. Let us notice that the proof of this step gives the proof of Lemma~\ref{par_con_resu}.

\noindent{\it Step 3.} We finish with the proof of the claim \eqref{A11}. For simplicity, we assume first \eqref{row_nonzero} holds for all rows, i.e. $d_{ij} \neq 0$ for all $i,j=1,\dots,m$, see the general case at the end of this proof. We assume without loss of generality that 
\begin{eqnarray}\label{meme1}
\min_{i \neq j}-d_{ij}=1,~~\max_{i=1,\dots,m}d_{ii}=\frac{M}{2},~\text{where $M$ is a constant bigger than $2m-2$}.
\end{eqnarray} 
\noindent{\it 3.1.} We first prove the claim for $R_1$. Let $t_0>0$ and $\phi \in C^1(0,\infty)$ such that $R_1-\phi$ attains a maximum at $t_0$ and suppose that $R_1(t_0)=N_1(t_0)$. Since $N_1(t_0) \ge N_j(t_0),~~j\ge 2$, then $\rho_{1j}(t_0)=1,~~j\ge 2$. Thus
\begin{eqnarray*}
&&\phi'(t_0)+\sum_{j=2}^{m}-d_{1j}R_1(t_0)+\sum_{j=2}^{m}d_{1j}N_j(t_0)\le 0, \text{i.e.,}\\
&&\phi'(t_0)+(m-1)R_1(t_0)-\sum_{j=2}^{m}N_j(t_0)+\sum_{j=2}^{m}(1+d_{1j})[N_j(t_0)-R_1(t_0)]\le 0.
\end{eqnarray*}
Using \eqref{meme1} and \eqref{reorder}, we have $1+d_{1j}\le 0,~~N_j(t_0)-R_1(t_0)\le 0,~~j \ge 2$. The above inequality with the fact that $\sum_{j=2}^{m}N_j(t_0)=\sum_{j=2}^{m}R_j(t_0)$ lead to
\begin{eqnarray}\label{equa1}
\phi'(t_0)+(m-1)R_1(t_0)-\sum_{j=2}^{m}R_j(t_0)\le 0.
\end{eqnarray}
\noindent{\it 3.2.} We now prove the claim for $R_k,~~k \ge 2$. Let $t_0>0$ and $\phi \in C^1(0,\infty)$ such that $R_k-\phi$ attains a maximum at $t_0$. Suppose first that $N_1(t_0)=R_k(t_0)<R_{k-1}(t_0)$, so $N_1(t_0)-\phi(t_0)= R_k(t_0)-\phi(t_0)\ge N_1(t)-\phi(t)$ for $t$ near $t_0$. It follows from the definition of $\rho_{ij}(.)$ that
\begin{eqnarray*}
\phi'(t_0)+\sum_{j=2}^{m}d_{1j}\rho_{ij}(t_0)[N_j(t_0)-R_k(t_0)]\le 0,~~\text{i.e.,}
\end{eqnarray*}
\begin{eqnarray}\label{bai2_in1}
\phi'(t_0)+\sum_{j \in \mathcal{I}}d_{1j}[N_j(t_0)-R_k(t_0)]+2\sum_{j \in \mathcal{I}^c}d_{1j}[N_j(t_0)-R_k(t_0)]\le 0,
\end{eqnarray}
where $\mathcal{I}:=\{2 \le j \le m,~~R_k(t_0)\ge N_j(t_0) \}$ and $\mathcal{I}^c=\{2,\dots,m\}-\mathcal{I}$. Since $N_1(t_0)=R_k(t_0)<R_{k-1}(t_0)$, we obtain that ${\rm card}(\mathcal{I}^c)=k-1$. We have
\begin{eqnarray*}
\sum_{j \in \mathcal{I}}d_{1j}[N_j(t_0)-R_k(t_0)]&=&{\rm card}(\mathcal{I})R_k(t_0)-\sum_{j \in \mathcal{I}}N_j(t_0)+\sum_{j \in \mathcal{I}}(1+d_{1j})[N_j(t_0)-R_k(t_0)]\\
&\ge&(m-k)R_k(t_0)-\sum_{j \in \mathcal{I}}N_j(t_0)=(m-k)R_k(t_0)-\sum_{j=k+1}^{m}R_j(t_0),
\end{eqnarray*}
where the last inequality follows from the fact $1+d_{1j}\le 0$ and $N_j(t_0)-R_1(t_0)\le 0,~~j \in \mathcal{I}$. 
Noticing that ${\rm card}(\mathcal{I}^c)=k-1$, we have
\begin{eqnarray*}
&& 2\sum_{j \in \mathcal{I}^c}d_{1j}[N_j(t_0)-R_k(t_0)]\\
&=&M(k-1)R_k(t_0)-M\sum_{j \in \mathcal{I}^c}N_j(t_0)+\sum_{j \in \mathcal{I}^c}(M+2d_{1j})[N_j(t_0)-R_k(t_0)]\\
&\ge&M(k-1)R_k(t_0)-M\sum_{j \in \mathcal{I}^c}N_j(t_0)=M(k-1)R_k(t_0)-M\sum_{j=1}^{k-1}R_j(t_0),
\end{eqnarray*}
where the inequality follows from the fact $M+2d_{1j}\ge 0$ and $N_j(t_0)-R_1(t_0)\ge 0,~~j \in \mathcal{I}^c$.
Using these two above inequalities, it follows from \eqref{bai2_in1} that
\begin{eqnarray}\label{equak}
\phi'(t_0)+[M(k-1)+m-k]R_k(t_0)-M\sum_{j=1}^{k-1}R_j(t_0)-\sum_{j=k+1}^{m}R_j(t_0)\le 0.
\end{eqnarray}

It remains to deal with the case $N_1(t_0)=R_k(t_0)=R_{k-1}(t_0)$. We divide it into two subcases

\noindent{\it 3.3.} If $N_1(t_0)=R_k(t_0)=R_{k-1}(t_0)=...=R_l(t_0)<R_{l-1}(t_0)$ with $k-1 \ge l \ge 2$, then $N_1(t_0)-\phi(t_0)= R_l(t_0)-\phi(t_0)\ge N_1(t)-\phi(t)$. Applying the result in \eqref{equak}, we have
\begin{eqnarray*}
\phi'(t_0)+[M(l-1)+m-l]R_l(t_0)-M\sum_{j=1}^{l-1}R_j(t_0)-\sum_{j=l+1}^{m}R_j(t_0)\le 0.
\end{eqnarray*}
It is clear that
\begin{eqnarray*}
&&[M(k-1)+m-k]R_k(t_0)-M\sum_{j=1}^{k-1}R_j(t_0)-\sum_{j=k+1}^{m}R_j(t_0)\\
&\le& [M(l-1)+m-l]R_l(t_0)-M\sum_{j=1}^{l-1}R_j(t_0)-\sum_{j=l+1}^{m}R_j(t_0).
\end{eqnarray*}
It follows that \eqref{equak} holds in this case too.

\noindent{\it 3.4.} If $N_1(t_0)=R_k(t_0)=R_1(t_0)$, then we have the estimate \eqref{equa1}. It is clear that
\begin{eqnarray*}
&&[M(k-1)+m-k]R_k(t_0)-M\sum_{j=1}^{k-1}R_j(t_0)-\sum_{j=k+1}^{m}R_j(t_0)\\
&\le& [m-1]R_1(t_0)-\sum_{j=2}^{m}R_j(t_0).
\end{eqnarray*}
Then \eqref{equak} holds.

\noindent{\it Step 4. }If \eqref{row_nonzero} only holds for a row, we will take the minimum in \eqref{meme1} among the $d_{ij}$'s which are nonzero and we keep zero elements of the coupling. Proceeding in a similar way as above, we obtain a new coupling satisfying \eqref{S3} and \eqref{row_nonzero}.

\end{proof}
\section{Proof of Theorem \ref{exist_smoo}}\label{strict}
We reuse mainly the ideas used in the proof of Theorem \ref{mainresult}. We then fix a Lipschitz solution $V$ of \eqref{HJEista} such that $u_i+c_1t-V_i \ge 2$, where $u$ is the solution of \eqref{HJEi} and the associated ergodic constant $c=(c_1,\dots,c_1)$. Withous loss of generality, we can assume that $c_1=0$. We can find, for all $\delta>0$, a function $V^\delta \in C^1(\T^N)^m$ such that
\begin{eqnarray}\label{HJEista_ep}
H_i(x, DV^\delta_i )+\sum_{j=1}^{m}d_{ij}V^\delta_j\le \delta,~~\text{and}~~||V^\delta_i-V_i||_\infty \le \delta,~~\text{for}~~i=1,\dots,m.
\end{eqnarray}
The existence of $V^\delta$ is obtained by the convolution of $V$ with a standard  mollifier. It is worth noticing that the convexity of $H$ is important and that $V^\delta$ is still a $\T^N$ periodic function. 

 Similarly as in the proof of Theorem \ref{mainresult}, we perform the change of function $\exp(w_i^\delta)=u_i-V^\delta_i$. The function $w^\delta$ is the solution of the new system
\begin{eqnarray}\label{inter1_ep}
&&\frac{\partial w_i^\delta}{\partial t}+F_i^\delta(x,w_i^\delta,Dw_i^\delta )+\sum_{j=1}^{m}d_{ij}{\rm exp}(w_j^\delta-w_i^\delta)\\
&+&{\rm exp}(-w_i^\delta)[H_i(x,DV^\delta_i)+ \sum_{j=1}^{m}d_{ij}V^\delta_j]=0.\nonumber
\end{eqnarray} 
with $F_i^\delta(x,w,p)= {\rm exp}(-w)[H_i(x,{\rm exp}(w)p +DV^\delta_i)-H_i(x,DV^\delta_i)]$. We can check $F_i^\delta$ satisfies \eqref{H13} with $\K= \emptyset$. Moreover, the term $\psi(\eta)$ appearing in \eqref{H13} can be chosen independently with $\delta$. The proof relies on the upper semicontinuity of the subdifferentials of convex functions and the strict convexity of the Hamiltonians. The concrete calculation is left to the reader.

We choose a sequence $t_n\to +\infty$ such that $(u(\cdot,t_n+\cdot))_{n}$ converges uniformly to
some function in $W^{1,\infty}(\T^N\times [0,\infty))^m$. Thus $(w^\delta(\cdot,t_n+\cdot))_{n} \to v^\delta\in W^{1,\infty}(\T^N\times [0,\infty))^m$  for all $\delta>0$. Set $\exp(w_i)=u_i-V_i$, we also have $(w(\cdot,t_n+\cdot))_{n} \to v\in W^{1,\infty}(\T^N\times [0,\infty))^m$. It is clear that $v^\delta \to v$ uniformly as $\delta \to 0$. 
We state the key result on these functions. 
\begin{lem}\label{IM4_ep}
Under the conditions of Theorem \ref{exist_smoo} and with $P_{\eta}$'s defined in \eqref{newfunc}, we have
 
$(i)$~~$P_{\eta}[w_i](t)$ converges to the same limit, as a result $P_{\eta}[v_i](.)=c(\eta)$ for all $i=1,\dots,m$. 

$(ii)$~~The $P_{\eta}[w_i](t)$'s attain their maximum at the same point, see \eqref{tothink} for the definition.

$(iii)$ $c(\eta)=0$ for any $\eta >0$.

\end{lem}

\begin{proof}[Proof of Theorem \ref{exist_smoo}]
Set  $S(t)u_0=u(x,t)$, it is clear that the semigroup $S(t)$ is nonexpansive in
$C(\T^N;\R^m)$. It is then sufficient to show the result for $u_0\in W^{1,\infty}(\T^N)^m$. Using Proposition \ref{PrReg} for $u_0\in W^{1,\infty}(\T^N)$, we get that $u \in W^{1,\infty}(\T^N \times [0,\infty))^m$. Now, having Lemma \ref{IM4_ep} $(iii)$ in hands, we repeat readily the proof of Theorem \ref{mainresult} to obtain the convergence as desired.
\end{proof}

 We end this section with the proof of Lemma \ref{IM4_ep}.

\begin{proof}[Proof of Lemma \ref{IM4_ep}]
{\it Proof of part $(i)$.} 

\noindent{\it Step 1.} We show that the $P_{\eta}[w_i^\delta]$'s are a subsolution of
\begin{eqnarray}\label{T11}
P'_{\eta}[w_i^\delta]+\sum_{j=1,~~j \neq i}^{m}-d_{ij}\rho^\delta_{ij}(\tau)\big\{1-{\rm exp}[P_{\eta}[w_j^\delta]-P_{\eta}[w_i^\delta]]\big\}-\delta\le 0,~~i=1,\dots,m,
\end{eqnarray}
where $\rho^\delta_{ij}$ are functions defined by
\begin{eqnarray}\label{def of rho2}
\rho^\delta_{ij}(t)=
&& \left\{
\begin{array}{ll}
m_1~~\text{if $P_{\eta}[w_i^\delta](t)\ge P_{\eta}[w_j^\delta](t)$},\\
m_2~~\text{if $P_{\eta}[w_i^\delta](t)< P_{\eta}[w_j^\delta](t)$},
\end{array}
\right.
\end{eqnarray}
for all $i\neq j$, $t>0,$ and $m_1\le m_2$ are positive constants. To do so, let $\Phi \in C^1((0,\infty))$ and $\tau$ be a strict maximum point of $P_{\eta}[w_i^\delta]-\Phi$ over $[\tau-\delta,\tau+\delta]$ for some $\delta >0$.  Argue exactly as Step 1 in the proof of Lemma \ref{IM2}, we may assume that $P_{\eta}[w_i^\delta](\tau)> 0$ to continue.

\noindent{\it Step 2}. Consider, $x,y \in \T^N,t \in [\tau-\delta,\tau+\delta]$ and $s \ge t$ the test function
\begin{eqnarray*}
\Psi(x,y,t,s)=w^\delta_i(x,t)-w^\delta_i(y,s)-2\eta(s-t)-\frac{|x-y|^2}{2\epsilon^2}-\Phi(t).
\end{eqnarray*}
Assume that $\Psi$ achieves its maximum over $\T^N \times \T^N \times  \{(t,s) / t \le s ,t \in [\tau-\delta,\tau+\delta]$\} at  $(\bar{x},\bar{y},\bar{t},\bar{s})$. We obtain some classical estimates when $\epsilon \to 0$,
\begin{eqnarray}\label{Esthu}
&& \left\{
\begin{array}{ll}
\Psi(\bar{x},\bar{y},\bar{t},\bar{s}) \to P_{\eta}[w_i^\delta](\tau)-\Phi(\tau),~~\frac{|\bar{x}-\bar{y}|^2}{2\epsilon^2} \to 0,\\ [1mm]
~~\text{$\bar{t} \to \tau$ since $\tau$ is a strict maximum point of $P_{\eta}[w_i^\delta]-\Phi$ in $[\tau-\delta,\tau+\delta]$},\\
w^\delta_i(\bar{x},\bar{t})-w^\delta_i(\bar{y},\bar{s}) \ge P_{\eta}[w_i^\delta](\bar{t}),~~\bar{s}>\bar{t} \quad \text{since} \quad P_{\eta}[w_i^\delta](\tau)> 0.
\end{array}
\right.
\end{eqnarray}

\noindent{\it Step 3}. Since $w^\delta$ is the solution of \eqref{inter1_ep}, we have
\begin{eqnarray}\label{A12}
&& \left\{
\begin{array}{ll}
\Phi'(\bar{t})-2\eta+F^\delta_i(\bar{x},w^\delta_i(\bar{x},\bar{t}),\frac{\bar{x}-\bar{y}}{\epsilon^2})+\sum_{j=1}^{m}d_{ij}{\rm exp}(w^\delta_j-w^\delta_i)(\bar{x},\bar{t})+a(\bar{x},\bar{t})\le 0,\\
-2\eta+F^\delta_i(\bar{y},w^\delta_i(\bar{y},\bar{s}),\frac{\bar{x}-\bar{y}}{\epsilon^2})+\sum_{j=1}^{m}d_{ij}{\rm exp}(w^\delta_j-w^\delta_i)(\bar{y},\bar{s})+a(\bar{y},\bar{s})\ge 0,
\end{array}
\right.
\end{eqnarray}
where
\begin{eqnarray}\label{axy}
a(x,t):={\rm exp}(-w_i^\delta(x,t))[H_i(x,DV^\delta_i(x))+ \sum_{j=1}^{m}d_{ij}V^\delta_j(x)].
\end{eqnarray}
We  estimate,
\begin{eqnarray*}
a(\bar{x},\bar{t})-a(\bar{y},\bar{s})&=&I_1+I_2\\
&=& [{\rm exp}(-w_i^\delta(\bar{x},\bar{t}))-{\rm exp}(-w_i^\delta(\bar{y},\bar{s}))][H_i(\bar{x},DV^\delta_i(\bar{x}))+ \sum_{j=1}^{m}d_{ij}V^\delta_j(\bar{x})]\\
&+& {\rm exp}(-w_i^\delta(\bar{y},\bar{s}))[H_i(\bar{x},DV^\delta_i(\bar{x}))-H_i(\bar{y},DV^\delta_i(\bar{y}))+ \sum_{j=1}^{m}d_{ij}(V^\delta_j(\bar{x})-V^\delta_j(\bar{y}))].
\end{eqnarray*}
Since $0 \ge {\rm exp}(-w_i^\delta(\bar{x},\bar{t}))-{\rm exp}(-w_i^\delta(\bar{y},\bar{s})) \ge -1$, and
$V^\delta$ is a subsolution of~\eqref{HJEista_ep}, we get that $I_1=[{\rm exp}(-w_i^\delta(\bar{x},\bar{t}))-{\rm exp}(-w_i^\delta(\bar{y},\bar{s}))][H_i(\bar{x},DV^\delta_i(\bar{x}))+ \sum_{j=1}^{m}d_{ij}V^\delta_j(\bar{x})] \ge -\delta$.
It is clear that $I_2$ tends to 0 when $\epsilon\to 0$ since $V^\delta \in C^1(\T^N)^m$. Hence, 
\begin{eqnarray}\label{rev_estim}
a(\bar{x},\bar{t})-a(\bar{y},\bar{s})\ge -\delta+o_\e(1).
\end{eqnarray}

Using \eqref{H13} and the fact that $w_i(\bar{x},\bar{t}) \ge w_i(\bar{y},\bar{s})\ge M_{\eta,i}(\bar{t}) \ge 0$, we have
\begin{eqnarray*}
F^\delta_i(\bar{x},w^\delta_i(\bar{x},\bar{t}),\frac{\bar{x}-\bar{y}}{\epsilon^2})-F^\delta_i(\bar{x},w^\delta_i(\bar{y},\bar{s}),\frac{\bar{x}-\bar{y}}{\epsilon^2}) \ge 0.
\end{eqnarray*}
Moreover,
\begin{eqnarray*}
&&\sum_{j=1}^{m}d_{ij}{\rm exp}(w^\delta_j-w^\delta_i)(\bar{x},\bar{t})-\sum_{j=1}^{m}d_{ij}{\rm exp}(w^\delta_j-w^\delta_i)(\bar{x},\bar{s})\\
&\ge& \sum_{j \neq i}-d_{ij}{\rm exp}(w^\delta_j-w^\delta_i)(\bar{x},\bar{s})\big\{1-{\rm exp}[P_{\eta}[w_j^\delta](\tau)-P_{\eta}[w_i^\delta](\tau)]\big\}+o_\e(1)\\
&\ge& \sum_{j \neq i}-d_{ij}\rho^\delta_{ij}(\tau)\big\{1-{\rm exp}[P_{\eta}[w_j^\delta](\tau)-P_{\eta}[w_i^\delta](\tau)]\big\}+o_\e(1),
\end{eqnarray*}
where $\rho^\delta_{ij}$ is defined as in~\eqref{def of rho2}
with $$m_1=\inf_{x \in \T^N,~~s>0,~~\delta>0,~~1 \le i,j \le m}{\rm exp}(w^\delta_j-w^\delta_i)(x,s)>0,$$ $$m_2=\sup_{x \in \T^N,~~s>0,~~\delta>0,~~1 \le i,j \le m}{\rm exp}(w^\delta_j-w^\delta_i)(x,s)<\infty,$$
which are well-defined thanks to the boundedness of $w_i$.

Therefore, by letting $\epsilon$ tend to 0 after subtracting both sides in \eqref{A12}, we get
\begin{eqnarray*}
\Phi'(\tau)+\sum_{j=1,~~j \neq i}^{m}-d_{ij}\rho^\delta_{ij}(\tau)\big\{1-{\rm exp}[P_{\eta}[w_j^\delta](\tau)-P_{\eta}[w_i^\delta](\tau)]\big\}-\delta\le 0.
\end{eqnarray*}
\noindent{\it Step 4}. We deduce easily that $N^\delta_{\eta,j}:={\exp(P_{\eta}[w_j^\delta])}$ is a subsolution of
\begin{eqnarray*}
N'^\delta_{\eta,i}+\sum_{j=1,j \neq i}^{m}d_{ij}\rho^\delta_{ij}[N^\delta_{\eta,j}-N^\delta_{\eta,i}]-C \delta \le 0,~~i=1,\dots,m.
\end{eqnarray*}

Since $N^\delta_{\eta,i}\to N_{\eta,i}:={\exp(P_{\eta}[w_i])}$ as $\delta \to 0$, by the stability result, $N_{\eta,i}$ is a subsolution of
\begin{eqnarray*}
N'_{\eta,i}+\sum_{j=1,j \neq i}^{m}d_{ij}\rho_{ij}[N_{\eta,j}-N_{\eta,i}]\le 0,~~i=1,\dots,m,
\end{eqnarray*} 
where $\rho_{ij}$ are functions defined by
\begin{eqnarray*}
\rho_{ij}(t)=
&& \left\{
\begin{array}{ll}
m_1~~\text{if $P_{\eta}[w_i](t)\ge P_{\eta}[w_j](t)$},\\
m_2~~\text{if $P_{\eta}[w_i](t)< P_{\eta}[w_j](t)$}.
\end{array}
\right.
\end{eqnarray*}
Thanks to Lemma \ref{IM2}, $N_{\eta,i}(t)$ converges to the same limit as $t$ tends to infinity. Hence $P_{\eta}[v_i](t)=c(\eta)$, where $c(\eta)$ is independent of $i$ and $t$.

{\it Proof of part $(ii)$.}

\noindent{\it Step 5.} Fix $\tau>0$. If $P_{\eta}[v_i](\tau)= 0$, then we finish the proof since we can choose $s_\tau=\tau$ and any $x_\tau$ to fulfill the requirement. We then assume that $P_{\eta}[v_i](\tau)> 0$ and that $P_{\eta}[v_i](\tau)$ attains its maximum at $x_\tau,~~s_\tau$. Consider, $x,y \in \T^N,t \in(0,\infty)$ and $s \ge t$ the test function
\begin{eqnarray*}
\Psi(x,y,t,s)=v^\delta_i(x,t)-v^\delta_i(y,s)-2\eta(s-t)-|x-x_\tau|^2-|t-\tau|^2-|s-s_\tau|^2-\frac{|x-y|^2}{2\epsilon^2}.
\end{eqnarray*}
The function $\Psi^{i,\epsilon}$ achieves its maximum over $\T^N \times \T^N \times  \{(t,s) / 0 \le t \le s $\} at  $(\bar{x},\bar{y},\bar{t},\bar{s})$. We obtain some classical estimates,
\begin{eqnarray*}
&& \left\{
\begin{array}{ll}
\frac{|\bar{x}-\bar{y}|^2}{2\epsilon^2} \to 0~~\text{ when $\epsilon \to 0$},\\

{\rm lim}_{\delta \to 0}{\rm lim}_{\e \to 0}(\bar{x},\bar{s},\bar{t})=(x_\tau,s_\tau,\tau),\\

v^\delta_i(\bar{x},\bar{t})\ge v^\delta_i(\bar{y},\bar{s}),~~\bar{s}>\bar{t} \quad \text{for $\e,~~\delta$ are small enough since}~~P_{\eta}[v_i](\tau)> 0.
\end{array}
\right.
\end{eqnarray*}
{\it Step 6.} Since $v^\delta$ is the solution of \eqref{inter1_ep}, we have
\begin{eqnarray}\label{Es21}
&& \left\{
\begin{array}{ll}
-2\eta+F^\delta_i(x_\tau,v^\delta_i(\bar{x},\bar{t}),\frac{\bar{x}-\bar{y}}{\epsilon^2})+\sum_{j=1}^{m}d_{ij}{\rm exp}(v^\delta_j-v^\delta_i)(\bar{x},\bar{t})
+a(\bar{x},\bar{t})\le o_\e(1),\\
-2\eta+F^\delta_i(x_\tau,v^\delta_i(\bar{y},\bar{s}),\frac{\bar{x}-\bar{y}}{\epsilon^2})+\sum_{j=1}^{m}d_{ij}{\rm exp}(v^\delta_j-v^\delta_i)(\bar{x},\bar{s})+a(\bar{y},\bar{s}) \ge o_\e(1),
\end{array}
\right.
\end{eqnarray}
where $a(x,t)$ is defined in \eqref{axy}. Arguing as Step 4 in the proof of Theorem \ref{IM1} with taking \eqref{rev_estim} into account (let $\epsilon$ tend to 0 and then $\delta$ to 0), we obtain that $Q_{\eta,j}(\tau)=P_{\eta}[v_i](\tau)$, for all $j$, with $Q_{\eta,j}(\tau):=v_j(x_\tau,\tau)-v_j(x_\tau,s_\tau)-2\eta(s_\tau-\tau)$.

\noindent {\it Proof of part $(iii)$.} We repeat Step 5, 6 with a small refinement. For any fixed $\tau>0$, we call $(x_\tau,s_\tau)$ the common minimum point of the $P_{\eta}[v_i](\tau)$'s. We choose $i \in \{1,\dots,m\}$ such that 
\begin{eqnarray*}
v_i(x_\tau,s_\tau)=\min_{j=1,\dots,m}v_j(x_\tau,s_\tau).
\end{eqnarray*}
Since ${\rm lim}_{\delta \to 0}{\rm lim}_{\e \to 0}(\bar{x},\bar{s})=(x_\tau,s_\tau)$, we get
\begin{eqnarray*}
\sum_{j=1}^{m}d_{ij}{\rm exp}(v_j-v_i)(\bar{x},\bar{s})\le o_\e(1)+o_\delta(1).
\end{eqnarray*}
This gives
\begin{eqnarray*}
F^\delta_i(\bar{x},v^\delta_i(\bar{y},\bar{s}),\frac{\bar{x}-\bar{y}}{\epsilon^2}) \ge \eta/2>0~~\text{for $\epsilon,~~\delta$ are small enough}.
\end{eqnarray*}
From \eqref{H13}$(ii)$, we obtain
\begin{eqnarray*}
F^\delta_i(\bar{x},v^\delta_i(\bar{x},\bar{t}),\frac{\bar{x}-\bar{y}}{\epsilon^2})-F^\delta_i(\bar{x},v^\delta_i(\bar{y},\bar{s}),\frac{\bar{x}-\bar{y}}{\epsilon^2}) \ge \psi(\eta)(v^\delta_i(\bar{x},\bar{t})-v^\delta_i(\bar{y},\bar{s})).
\end{eqnarray*}
Recall that $\psi(\eta)$ does not depend on $\delta$. With the same computation as \eqref{T1}, we have
\begin{eqnarray*}
\sum_{j=1}^{m}d_{ij}{\rm exp}(v^\delta_j-v^\delta_i)(\bar{x},\bar{t})-\sum_{j=1}^{m}d_{ij}{\rm exp}(v^\delta_j-v^\delta_i)(\bar{x},\bar{s})\ge o_\e(1)+o_\delta(1).
\end{eqnarray*}
Therefore, subtracting both sides in \eqref{Es21} and taking \eqref{rev_estim} into account, we get
\begin{eqnarray*}
\psi(\eta)(v^\delta_i(\bar{x},\bar{t})-v^\delta_i(\bar{y},\bar{s}))\le o_\e(1)+o_\delta(1).
\end{eqnarray*}
Letting $\epsilon$ tend to 0 and then $\delta$ to 0, we obtain that $\psi(\eta)c(\eta)\le 0$. It is a contradiction.
\end{proof}

\section{Proof of Theorem \ref{largenew}}\label{nr}
With $\mathcal{S}$ defined in \eqref{samemin}, we will show that 
\begin{eqnarray}\label{erco}
c_1=-\min_{x \in \T}\frac{\sum_{i=1}^m \Lambda_i f_i(x)}{\sum_{i=1}^m \Lambda_i}=-\frac{\sum_{i=1}^m \Lambda_i f_i(x_0)}{\sum_{i=1}^m \Lambda_i}~~\text{for any $x_0 \in \mathcal{S}$},
\end{eqnarray}
where the vector $\Lambda=(\Lambda_1,\dots,\Lambda_m) >0$ is given by Lemma~\ref{lem-rang}. To do so, we need
\begin{lem}\label{del_inver}
For any matrix $D$ satisfying \eqref{irr}, the matrix $E$ obtaining from $D$ after canceling the $m$th row and $m$th column is invertible.
\end{lem}
\begin{proof}[Proof of Lemma \ref{del_inver}]
Assume by contradiction that there exists $x=(x_1,...,x_{m-1}) \neq 0$ such that $Ex=0$. It is clear that $y=(x,0)$ satisfies $Dy=0$, then Lemma \ref{lem-rang} yields that $x=0$. This is a contradiction.
\end{proof}
\begin{lem}\label{construct}
There exists a constant solution of the system 
\begin{eqnarray}\label{constant}
\sum_{j=1}^{m}d_{ij}u_j=b_i-\frac{\sum_{i=1}^m \Lambda_i b_i}{\sum_{i=1}^m \Lambda_i},
\quad i=1,\dots,m.
\end{eqnarray}
\end{lem}
\begin{proof}[Proof of Lemma \ref{construct}]
Set $a:=\frac{\sum_{i=1}^m \Lambda_i b_i}{\sum_{i=1}^m \Lambda_i}.$ From Lemma \ref{del_inver}, we can find $(u_1,...,u_{m-1})$ satisfying
\begin{eqnarray*}
\sum_{j=1}^{m-1}d_{ij}u_j=b_i-a,~~ i=1,\dots,m-1.
\end{eqnarray*}
Set $u_m=0$, we have
\begin{eqnarray}\label{m-1}
\sum_{j=1}^{m}d_{ij}u_j=b_i-a,~~ i=1,\dots,m-1.
\end{eqnarray}
We claim that \eqref{m-1} holds for $i=m$. Multiplying the $i$th equation in \eqref{m-1} 
by $\Lambda_i$ and summing all equations for $i=1,\dots,m-1$, we obtain 
\begin{eqnarray*}
\sum_{i=1}^{m-1} (\sum_{j=1}^{m}\Lambda_i d_{ij}u_j)
=\sum_{i=1}^{m-1}\Lambda_i b_i-\sum_{i=1}^{m-1}\Lambda_i a,
\quad \text{i.e,} \quad
\sum_{j=1}^{m}-\Lambda_m d_{mj} u_j=\Lambda_m(a- b_m),
\end{eqnarray*}
which is exactly what we need.
\end{proof}
From the uniqueness of the ergodic constant and thanks to Lemma \ref{construct}, we obtain the formula \eqref{erco} for the ergodic constant. We now prove Theorem \ref{largenew}.


\begin{proof}[Proof of Theorem \ref{largenew}]
Arguing as in the proof of Theorem \ref{exist_smoo},  we can assume that $u \in W^{1,\infty}(\T^N \times [0,\infty))^m$. From Lemma \ref{construct}, there exists a constant solution $(a_1,...,a_m)$ for~\eqref{constant} with $b_i=f_i(x_0),~~x_0 \in \mathcal{S}$ for $i=1,\dots,m$. So, $w_i(x,t)=u_i(x,t)+c_1t-a_i$ solves
\begin{eqnarray*}
\frac{\partial w_i}{\partial t}+  F_i(x, Dw_i )+\sum_{j=1}^{m}d_{ij}w_j=f_i(x)-f_i(x_0)
& (x,t)\in \T^N\times (0,+\infty),~~x_0\in \mathcal{S},~~i=1,\dots,m.
\end{eqnarray*}
Since this new system satisfies the conditions of Theorem \ref{mainresult}, we get the convergence of $w_i(.,t)$ and hence of $u_i(.,t)$.
\end{proof}

\section{Proof of Theorem \ref{identical}}\label{bs2}

We require $D$ to be {\em nonzero}, i.e.,
\begin{eqnarray}\label{nonzero}
\text{For any~~$i,j = 1,\dots,m$, there exists $k\in\{1,\cdots ,m\}$ such that $d_{ik}d_{jk}\neq 0$.}
\end{eqnarray}
The main consequence of systems whose hamiltonians are identical is the following result, the proof of which is given in Appendix.
\begin{lem}\label{samelimit}
Under the assumptions of Theorem \ref{identical}, we have
\begin{eqnarray*}
\limsup_{t \to \infty}\max_{x \in \T^N,~~i,j=1,\dots,m}|u_i(x,t)-u_j(x,t)| =0.
\end{eqnarray*}
\end{lem}

\begin{proof}[Proof of Theorem \ref{identical}]
For each $n \ge 0$, call $\Phi^n$ be the solution to the equation: 
\begin{equation}\label{evo_equa}
\left\{
  \begin{array}{ll}\displaystyle
\frac{\partial \Phi^n}{\partial t}+  H(x, D\Phi^n )=0
   & (x,t)\in  \T^N\times (0,+\infty), \\[5pt]
  \Phi^n(x,0)=u_1(x,n)&x\in \T^N,
  \end{array}
\right.
\end{equation}
Note that $(\Phi^n,\dots,\Phi^n)$ is a solution of \eqref{HJEi} with initial conditions $(u_1(.,n),\dots,u_1(.,n))$. Applying the comparison principle for the system \eqref{HJEi}, we obtain
\begin{eqnarray}
&&\sup_{i=1,\dots,m,~~x \in \T^N,~~t\ge 0}|u_i(x,t+n)-\Phi^n(x,t)| \le \sup_{1\le i \le m,~~x \in \T^N}|u_i(x,n)-u_1(x,n)|. \label{aa}\\
&&\sup_{x \in \mathbb{T}^N,~~t \ge 0}|\Phi^n(x,t)-\Phi^{n+1}(x,t)| \le \sup_{x \in \mathbb{T}^N}|u_1(x,n)-u_1(x,n+1)|. \label{cc}
\end{eqnarray}

From the convergence of the solutions of \eqref{evo_equa}, we have
\begin{eqnarray}
\Phi^n(.,t) \to V^n(.)~~\text{in}~~ C(\T^N)~~\text{as}~t \to \infty,~~ \text{for some} ~~V^n\in {\rm BUC}(\T^N)\label{ee}
\end{eqnarray}
and $V^n$ is a solution of \eqref{ergo_sing}. From \eqref{cc} and \eqref{ee} we infer that $(V^n)_n$ is a Cauchy sequence in ${\rm BUC}(\T^N)$ and therefore
\begin{eqnarray}\label{ff}
V^n(.) \to V(.)~~ \text{in}~~ C(\T^N),~~ \text{for some} ~~V\in {\rm BUC}(\T^N) .
\end{eqnarray}
By the stability result, $V$ is still a solution of \eqref{ergo_sing}. Using \eqref{ee}, we take $\limsup^*$ with respect to $t$ on both sides of \eqref{aa}
\begin{eqnarray*}
\overline{u_i}(x) \le V^n(x)+\sup_{i=1,\dots,m,~~x \in \T^N}|u_i(x,n)-u_1(x,n)|, ~~\text{for all $i$ and}~~ x \in \T^N.
\end{eqnarray*}
From \eqref{ff} and Lemma \ref{samelimit}, we let $n$ tend to infinity in the above inequality to obtain
\begin{eqnarray*}
\overline{u_i}(x) \le V(x)~~\text{for all $i$ and}~~ x \in \T^N.
\end{eqnarray*}
Similarly, we get $\underline{u_i}(x) \ge V(x)$. Hence, $u_i(.,t) \to V(.)$ as $t$ tends to infinity.
\end{proof}
\section{Appendix}
\subsection{The control interpretation.}
For the reader's convenience, we give a formal link between optimal
control of hybrid systems with pathwise deterministic
trajectories with random switching and
Hamilton-Jacobi systems~\eqref{HJEi} with convex Hamiltonians.

Consider the controlled random evolution process
$(X_t,\nu_t)$ with dynamics
\begin{equation}\label{REV}
\left\{
\begin{array}{l}
\dot X_t =  b_{\nu_t}(X_t ,a_t), \ \ t> 0,\\
(X_0,\nu_0)=(x,i) \in \T^N\times \{1,\dots,m\},
\end{array}
\right.
\end{equation}
where the control law  $a:[0,\infty)\to A$ is a measurable function
($A$ is a subset of some metric space),
$b_i\in L^\infty(\T^N\times A; \R^N),$  satisfies
\begin{eqnarray}
|b_i(x,a)-b_i(y,a)|\leq C|x-y|,\qquad x,y\in\T^N, \ a\in A, \ 1\leq i\leq m.
\end{eqnarray}
For every $a_t$ and matrix of probability transition
$G=(\g_{ij})_{i,j}$ satisfying $\sum_{j\not= i}\g_{ij}=1$
for $i\not= j$ and $\g_{ii}=-1,$
there exists a solution $(X_t,\nu_t),$ where $X_t:[0,\infty)\to \T^N$
is piecewise $C^1$ and $\nu(t)$  is a continuous-time Markov chain
with state space $\{1,\dots,m\}$ and probability transitions given by
\[
\P\{\nu_{t+\Delta t}=j\,|\,
\nu_t=i\}=\g_{ij}\Delta t+o(\Delta t)
\]
for $j\neq i.$

We introduce the value functions of the optimal control problems
\begin{equation}\label{Value}
    u_i(x,t)=\inf_{a_t\in L^\infty([0,t],A)}\E_{x,i}\big\{\int_0^t \ell_{\nu_s}(X_s,a_s)ds
+u_{0,\nu_t} (X_t)\big\},
\quad i=1,\dots m,
\end{equation}
where $\E_{x,i}$ denote the expectation of a trajectory starting at
$x$ in the mode $i,$ and the functions $u_{0,i}: \T^N\to \R$, $\ell_i: \T^N \times A\to \R$ are continuous.

It is possible to show that the following dynamic programming principle
holds:
\begin{eqnarray*}
u_i (x,t)= \mathop{\rm inf}_{a_t\in L^\infty([0,t],A)} \E_{x,i} \big\{
\int_0^h \ell_{\nu_s}(X_s,a_s)ds + u_{\nu_h}(X_h,t-h)
\big\} \qquad 0<h\leq t.
\end{eqnarray*}
Then the functions $u_i$ satisfy the system
\begin{equation*}
\left\{
  \begin{array}{ll}\displaystyle
\frac{\partial u_i}{\partial t}+ \mathop{\rm sup}_{a\in A}
[-\langle b_i(x,a), Du_i\rangle- \ell_i(x,a)]
+\sum_{j\not= i}\g_{ij}(u_i-u_j)=0
   & (x,t)\in  \T^N\times (0,+\infty), \\[5pt]
  u_{i}(x,0)=u_{0,i}(x)&x\in \T^N,
  \end{array}
i=1, \cdots m,
\right.
\end{equation*}
which has the form \eqref{HJEi} by setting
$H_i(x,p)= \mathop{\rm sup}_{a\in A}
[-\langle b_i(x,a), p\rangle-\ell_i(x,a)]$ and $d_{ii}=\sum_{j\not= i} \g_{ij}=1$
and $d_{ij}= -\g_{ij}$ for $j\not=i.$
\begin{rem} \ \\
(i) Assume $\ell_i(x,a)=f_i(x)$ where the $f_i$'s satisfy \eqref{samemin}. If the following controllability assumption is satisfied: 
for every $i,$
there exists $r>0$ such that for any $x\in\T^N$, the ball $B(0,r)$ is contained in $\overline{\textrm{co}}\{b_i(x,A)\}$.
Then, Theorem~\ref{largenew} holds. Roughly speaking, it means that
the optimal strategy is to drive the trajectories towards a point $x^*$ of
$\mathcal{S}$ and then not to move anymore (except maybe a small time before $t$).
This is suggested by the fact that all the $f_i$'s attain their minimum 
at $x^*$ and, at such point, the running cost is smallest.\\
(ii) It is also possible to consider differential games with random switchings to encompass system~~\eqref{HJEi} with nonconvex Hamiltonians.\\
(iii) More rigorous dynamical interpretations of system~~\eqref{HJEi} are given in \cite{mt12b}.
\end{rem}

\subsection{Proof of the ergodic problem. }
\begin{proof}[Proof of Theorem \ref{ergo3}]
\noindent{\it Step 1. Ergodic approximation.}
We consider the ergodic approximation to \eqref{HJEista}:
for  $\l \in (0,1)$, let  $v^\l=(v_1^\l,\dots,v^\l_m)$ be the solution  of
\begin{equation}\label{HJLm}
 \l v_i+ H_i(x,Dv_i)+\sum_{j=1}^{m}d_{ij}(x)v_j=0
\qquad x\in\T^N, \ 1\leq i\leq m.
\end{equation}

\begin{lem} {\rm (\cite[Lemma 4.1]{clln12})}
There exist a unique solution $v^\l$ of~\eqref{HJLm} and a constant $M>0$ independent of $\l$ such that
\begin{eqnarray}\label{est}
0\leq v_i^\l\leq \frac{M}{\l}
\quad {and} \quad
||Dv_i^\l(.)||_{\infty}\leq M,
\qquad i=1,\dots,m.
\end{eqnarray}
\end{lem}
\noindent{\it Step 2. Some uniform bounds.} We will prove at Step 4 that there exists a constant $M$ such that for all $x \in \T^N,~~i =1,\dots,m$, we have
\begin{eqnarray}\label{keyes}
|v_i^\l(x)-v_1^\l(x^*)|\le M~~\text{for any fixed $x^* \in \T^N$.}
\end{eqnarray}
From Ascoli's theorem, there exist
$c=(c_1,\dots,c_m)\in\R^m$ and $v\in C(\T^N)^m$
such that, up to subsequences, for $i=1,\dots,m$
\begin{eqnarray*}
\l v_i^\l(x^*)\to -c_i
\quad {\rm and} \quad
v^\l_i-v^\l_1(x^*)\to v_i \ {\rm in} \ C(\T^N)
~~\ {\rm as} \ \l\to 0.
\end{eqnarray*}
Notice that $c_i$ does not depend on the choice of $x^*$ since, for any $x^*,y^*\in\T^N$
\begin{eqnarray}\label{c-indep-x}
|-\l v_i^\l(x^*)+\l v_i^\l(y^*)|\leq \l M |x^*-y^*|\to 0~~{\rm as} \ \l\to 0.
\end{eqnarray}
Moreover, multiplying~\eqref{HJLm} by $\l$ for all $i$ and sending $\l\to 0,$
we obtain $-\sum_j d_{ij}(x)c_i=0$ which gives $D(x)c=0$ and therefore $c\in {\rm ker}\,D(x)$ for all $x\in \T^N$.

\noindent{\it Step 3. Stability result for viscosity solutions and conclusion.}
We rewrite \eqref{HJLm} as
\begin{eqnarray}\label{eq432}
&&\l v^\l_i+  H_i(x,D(v^\l_i-v^\l_1(x^*)))+\sum_{j=1}^m d_{ij}(v^\l_j-v^\l_1(x^*))=0
\quad {\rm in} \ \T^N,
\end{eqnarray}
by noting that $\sum_{j=1}^m d_{ij}(x)v^\l_1(x^*)=0$ for all $x \in \T^N$ thanks to \eqref{S3}.

We then use the stability result and pass to the limit in~\eqref{eq432} to get
\begin{eqnarray}
H_i(x,Dv_i)+\sum_{j=1}^m d_{ij}v_j(x)=c_i
\quad {\rm in} \ \T^N, \ 1\leq i\leq m.
\end{eqnarray}
Then $(c,v(\cdot))$ is solution to~\eqref{HJEista}. The function $v$ depends on $x^*$ but $c$ does not.

From Lemma~\ref{lem-rang}, the kernel of $D$ is spanned by $(1,\dots,1).$
Thus, any $c\in {\rm ker}\,D$ has the form $(c_1,\dots,c_1).$
The proof of uniqueness of $c$ is classical and can be found in \cite{clln12}.

\noindent{\it Step 4. Proof of \eqref{keyes}}.
We set 
\begin{eqnarray*}
w_i(x)=w_i^\l(x)=v_i^\l(x)-v_1^\l(x^*).
\end{eqnarray*} 
Thanks to~\eqref{est}, obviously $|w_1|\leq M.$
Using \eqref{est} again for \eqref{eq432}, we have
\begin{eqnarray*}
|\sum_{j=2}^m d_{ij}(x)w_j(x)|\le C ~~\text{for all $i=1,\dots,m$ and $x \in \T^N$},
\end{eqnarray*} 
where $C$ is independent of $x$ and $\lambda$.
For any $i \ge 2$, we have
\begin{eqnarray*}
d_{ii}|w_i|&=&|\sum_{j=2}^m d_{ij}w_j-\sum_{j=2,~~j \neq i}^m d_{ij}w_j|\le C+|\sum_{j=2,~~j \neq i}^m d_{ij}w_j|\le C-\sum_{j=2,~~j \neq i}^m d_{ij}|w_j|,
\end{eqnarray*} 
i.e.,
\begin{eqnarray}\label{est3}
\sum_{j=2}^m d_{ij}|w_j|\le C~~\text{for any $i \ge 2$}.
\end{eqnarray}
Call $(\Lambda_1,\dots,\Lambda_m)$ the function given by Lemma \ref{lem-rang}. We have
\begin{eqnarray*}
-\sum_{j=2}^m|w_j|d_{1j}\Lambda_1=
\sum_{j=2}^m|w_j|(\sum_{i=2}^m d_{ij}\Lambda_i)=\sum_{i=2}^m\Lambda_i\sum_{j=2}^m d_{ij}|w_j|\le C(\sum_{i=2}^m\Lambda_i).
\end{eqnarray*}
Assume $|w_2|=\min_{j \ge 2}|w_j|$, we have
\begin{eqnarray*}
d_{11}|w_2|=-\sum_{j=2}^md_{1j}|w_2|\le -\sum_{j=2}^m|w_j|d_{1j} 
\le C\frac{\sum_{i=2}^m\Lambda_i}{\Lambda_1}.
\end{eqnarray*} 

Thanks to \eqref{irr}, we have $d_{11}>0$. Using the compactness of $\T^N$ and continuity of the coupling, there exists $\delta_0>0$ such that $d_{11}(x)\ge \delta_0$ for all $x \in \T^N$. Therefore, we have
\begin{eqnarray*}
|w_2|\le C\frac{\sum_{i=2}^m\Lambda_i}{\Lambda_1 \delta_0}.
\end{eqnarray*} 
We finish the proof by a reduction argument, i.e., we assume that
\begin{eqnarray*}
\sum_{j=k}^m d_{ij}|w_j|\le C~~\text{for any $3 \le k \le m$ and $|w_l| \le C$ for $1 \le l \le k-1$},
\end{eqnarray*}
and we will show that $|w_k| \le C'$. By similar arguments like those to 
obtain the bound for $|w_2|$, we then assume that $|w_k|=\min_{j \ge k}|w_j|$.
We have
\begin{eqnarray*}
(\sum_{i=1}^{k-1}\Lambda_i\sum_{j=1}^{k-1}d_{ij})|w_k|=(-\sum_{j=k}^m\sum_{i=1}^{k-1}d_{ij}\Lambda_i)|w_k|\le-\sum_{j=k}^m|w_j|\sum_{i=1}^{k-1}d_{ij}\Lambda_i\le C(\sum_{i=k}^m\Lambda_i).
\end{eqnarray*} 
If $\sum_{i=1}^{k-1}\Lambda_i\sum_{j=1}^{k-1}d_{ij}(x)>0$ for all $x \in \T^N$, the conclusion follows easily by the compactness of $\T^N$ and 
the continuity of the coupling. We then assume by contradiction that $\sum_{i=1}^{k-1}\Lambda_i\sum_{j=1}^{k-1}d_{ij}(x_0)=0$ for some $x_0 \in \T^N$, \eqref{S3} yields $d_{ij}(x_0)=0$ for all $1 \le i \le k-1,~~k \le j \le m$. 
We get a contradiction with the choice
 $\mathcal{I}=\{1,\dots,k-1\}$ in \eqref{irr}.
\end{proof}

\subsection{Proof of Lemma \ref{samelimit}}
This proof is a modified version of the one in \cite{mt12} so that it can be adapted to general systems which is a little more tricky.

\noindent{\it Step 1. Some first estimates.} Thanks to \eqref{nonzero}, we have
\begin{eqnarray*}
\delta=\min_{x \in \T^N,~~i,j=1,\dots,m,~~\mathcal{I} \subset \{1,\dots,m\}}~~-[\sum_{k \in \mathcal{I}}d_{ik}(x)+\sum_{k \in \mathcal{I}^c}d_{jk}(x)]>0. 
\end{eqnarray*}
where $\mathcal{I}$ contains $j$ but not $i$.

Set $\Phi(t)=\max_{i \neq j,~~x \in \mathbb{T}^N}[u_i(x,t)-u_j(x,t)]\ge 0$ for each $t>0$. Our purpose is to prove that $\Phi$ is a subsolution to the equation
\begin{eqnarray} 
\Phi'(t)+\delta \Phi(t)=0.\label{ordina}
\end{eqnarray}
Assume without loss of generality that $\Phi(t)=u_1(x_0,t)-u_2(x_0,t)$ and all functions are smooth to do a formal proof. It can be done rigorously by approximation techniques.

We have $\Phi'(t)=\frac{\partial u_1}{\partial t}(x_0,t)-\frac{\partial u_2}{\partial t}(x_0,t),~~Du_1(x_0,t)=Du_2(x_0,t)$.
Subtracting two first equations in \eqref{HJEi}, we have 
\begin{eqnarray*}
\Phi'(t)+\sum_{j=1}^{m}d_{1j}(x_0)u_j(x_0,t)-\sum_{j=1}^{m}d_{2j}(x_0)u_j(x_0,t)=0.\\
\end{eqnarray*}
To obtain the conclusion, we only need to prove that 
\begin{eqnarray*}
\sum_{j=1}^{m}[d_{1j}(x_0)-d_{2j}(x_0)]u_j(x_0,t) \ge \delta(u_1(x_0,t)-u_2(x_0,t)) 
\end{eqnarray*}\\
or
\begin{eqnarray}\label{inequa}
(d_{11}-d_{21}-\delta)u_1 \ge (d_{22}-d_{12}-\delta)u_2+\sum_{j=3}^{m}(d_{2j}-d_{1j})u_j.
\end{eqnarray}
At the point $(x_0,t)$, we have 
\begin{eqnarray}
u_1 \ge u_3,...,u_m \ge u_2,~~d_{11}-d_{21}-\delta=d_{22}-d_{12}-\delta+\sum_{j=3}^{m}(d_{2j}-d_{1j})\label{stric}
\end{eqnarray}
but the signs of $d_{2j}-d_{1j}, j \ge 3$ are unknown. 

\noindent{\it Step 2. Separate the signs of $d_{2j}-d_{1j}, j \ge 3$ .} Let $J^+$ be the set consisting of all $j\in \{3,...,m\}$ such that $d_{2j}-d_{1j} \ge 0$ and $J^-:=\{3,...,m\}-J^+$. We rewrite \eqref{inequa} as
\begin{eqnarray*}
\quad(d_{11}-d_{21}-\delta)u_1-\sum_{j \in J^+}(d_{2j}-d_{1j})u_j \ge (d_{22}-d_{12}-\delta)u_2+\sum_{j \in J^-}(d_{2j}-d_{1j})u_j 
\end{eqnarray*}
Actually, we can prove a stronger inequality
\begin{eqnarray*}
\quad(d_{11}-d_{21}-\delta)u_1-\sum_{j \in J^+}(d_{2j}-d_{1j})u_1 \ge (d_{22}-d_{12}-\delta)u_2+\sum_{j \in J^-}(d_{2j}-d_{1j})u_2 
\end{eqnarray*}
It is clear by \eqref{stric} that 
\begin{eqnarray*}
d_{11}-d_{21}-\delta-\sum_{j \in J^+}(d_{2j}-d_{1j})=d_{22}-d_{12}-\delta +\sum_{j \in J^-}(d_{2j}-d_{1j})
\end{eqnarray*}
From this equality and $u_1 \ge u_2$, we only need to prove that 
\begin{eqnarray*}
d_{11}-d_{21}-\delta-\sum_{j \in J^+}(d_{2j}-d_{1j}) \ge 0.
\end{eqnarray*}
This is true by the definition of $\delta$.
\begin{eqnarray*}
d_{11}-d_{21}-\sum_{j \in J^+}(d_{2j}-d_{1j})= -[d_{12}+d_{21}+\sum_{j \in J^+}d_{2j}+\sum_{j \in J^-}d_{1j}] \ge \delta.
\end{eqnarray*}

Since $\Phi(0)e^{-\delta t}$ is a supersolution of \eqref{ordina} with the initial value $\Phi(0)$, the comparison principle yields $0 \le \Phi(t) \le \Phi(0)e^{-\delta t}$ for all $t$. Therefore, $\Phi(t)$ converges to 0 as $t \to \infty$.

\end{document}